\documentclass[10pt]{siamltex}

\usepackage{amsmath,amsfonts,amscd,amssymb,bm,cite}
\usepackage{mathrsfs,listings,url}
\usepackage{graphics,color}
\usepackage{float}
\usepackage[titletoc,page]{appendix}
\usepackage{subfigure}
\usepackage[colorlinks=true, bookmarksopen,
            pdfauthor={CST},
            pdfcreator={pdftex},
            pdfsubject={algorithms},
            linkcolor={blue},
            anchorcolor={black},
            citecolor={firebrick},
            filecolor={magenta},
            menucolor={black},
            pagecolor={red},
            plainpages=false,pdfpagelabels,
            urlcolor={db} ]{hyperref}
\usepackage{comment}
\usepackage{verbatim}
\usepackage{marginnote}
\usepackage{float}
\usepackage[makeroom]{cancel}
\usepackage[pdftex]{graphicx}
\usepackage[section]{placeins}
\usepackage{mathptmx}

\input{tcilatex}

\excludecomment{confidential}
\restylefloat{table}
\allowdisplaybreaks
\definecolor{db}{rgb}{0.0470,0,0.5294}
\definecolor{dg}{rgb}{0.0,0.392,0.0}
\definecolor{firebrick}{rgb}{0.698,0.133,0.133}
\definecolor{bl}{rgb}{0.0,0.0,0.0}
\definecolor{linen}{rgb}{0.980,0.941,0.902}
\definecolor{ivory}{rgb}{1.0,1.0,0.941}
\definecolor{aliceblue}{rgb}{0.941,0.973,1.0}
\definecolor{beige}{rgb}{0.961,0.961,0.863}
\definecolor{tan}{rgb}{0.824,0.706,0.549}
\definecolor{lightsteelblue}{rgb}{0.690,0.769,0.871}
\definecolor{paleturquoise}{rgb}{0.686,0.933,0.933}
\definecolor{lightblue}{rgb}{0.678,0.847,0.902}
\definecolor{skyblue}{rgb}{0.529,0.808,0.922}
\definecolor{palegoldenrod}{rgb}{0.933,0.910,0.667}
\definecolor{lightgoldenrod}{rgb}{0.933,0.867,0.510}
\definecolor{lightyellow}{rgb}{1.0,1.0,0.878}
\definecolor{yellow}{rgb}{1.0,1.0,0.0}
\definecolor{lightyellow1}{rgb}{1.0,1.0,0.878}
\definecolor{lemonchiffon}{rgb}{1.0,0.980,0.804}
\definecolor{myyellow}{rgb}{1,1,.9}
\definecolor{darkgreen}{rgb}{0.0,0.392,0.0}
\definecolor{darkviolet}{rgb}{0.580,0.0,0.827}
\definecolor{lightsalmon}{rgb}{1.0,0.627,0.478}
\definecolor{orange}{rgb}{1.0,0.647,0.0}

\numberwithin{equation}{section}

\begin{document}

\title
{Analysis of the variable step method of Dahlquist, Liniger and
Nevanlinna for fluid flow}
%
%
%
%
\author{William Layton\thanks{%
Department of Mathematics, University of Pittsburgh, Pittsburgh, PA 15260,
USA. Email: \href{mailto:wjl@pitt.edu}{wjl@pitt.edu}. The research herein was partially supported by NSF grant DMS 1817542.}
\and Wenlong Pei\thanks{%
Department of Mathematics, University of Pittsburgh, Pittsburgh, PA 15260,
USA. Email: \href{mailto:wep17@pitt.edu}{wep17@pitt.edu}.}
\and Yi Qin\thanks{%
School of Mathematics and Statistics, Xi'an Jiaotong University, Xi'an,
Shaanxi 710049, China. Email: \href{mailto:tqinyi1991@stu.xjtu.edu.cn}{qinyi1991@stu.xjtu.edu.cn}.}
\and Catalin Trenchea\thanks{%
Department of Mathematics, University of Pittsburgh, Pittsburgh, PA 15260,
USA. Email: \href{mailto:trenchea@pitt.edu}{trenchea@pitt.edu}. Partially supported by the NSF grant DMS-1522574.
}}


\date{\emty}
\maketitle



\begin{abstract}
The two-step time discretization proposed by Dahlquist, Liniger and
Nevanlinna is variable step $G$-stable. (In contrast, for increasing time
steps, the BDF2 method loses A-stability and suffers non-physical energy
growth in the approximate solution.) While unexplored, it is thus ideal for
time accurate approximation of the Navier-Stokes equations. This report
presents an analysis, for variable time-steps, of the method's stability and
convergence rates when applied to the NSE. It is proven that the method is
variable step,
unconditionally, long time stable and second order accurate. Variable step
error estimates are also proven. The results are supported by several
numerical tests.
\end{abstract}
\maketitle

\section{Introduction}

The accurate numerical simulation of flows of an incompressible, viscous
fluid, with the accompanying complexities occurring in practical settings,
is a problem where speed, memory and accuracy never seem sufficient. For
time discretization (considered herein), many simulations use the constant
step, first order, fully implicit method, e.g., Chen and Mclaughlin \cite
{MR3962839}, Jiang \cite{MR3260478},
Jiang and Tran \cite{MR3365766},
and (with few exceptions noted
in Section \ref{subsection:relatedwork}) the remainder use the constant
timestep trapezoid / implicit midpoint
scheme, e.g. Baker \cite{B76}, Baker, Dougalis and Karakashian
\cite{MR669634}, Ingram \cite{MR3055600}, Labovsky, Manica and Neda
\cite{MR2498863},  Simo, Armero and Taylor \cite{MR1331503}, (often
combined with fractional steps, Bristeau, Glowinski and P\'{e}riaux
\cite{BGP87} or with ad hoc fixes to correct
for oscillations due to lack of $L$-stability, ~\O sterby \cite{MR2068272}) or
the BDF2 method (e.g., Akbas, Kaya and Rebholz \cite{MR3652175}, Ascher and
Petzold \cite{MR1638643}, Grigorieff \cite{MR723632}, Mays and Neda
\cite{MR3264334}, Rong and Fiordilino \cite{RF18}).
Time \textit{accuracy} requires time step \textit{adaptivity} within the
computational, space and cognitive complexity limitations of CFD. Beyond
accuracy, adaptivity has the secondary benefit (depending on implementation)
of reducing memory requirements and decreasing the number of floating point
operations.

The richness of scales of higher Reynolds number flows and the cost per step
of their solution suggests a preference for $A$-stable (or even $L$-stable)
multi-step methods, called \textit{Smart Integrators} in  Gresho, Sani and
Engelman \cite[Section
3.16.4]{gresho1998incompressiblevol2}. For \textit{constant} time steps,
a complete analysis of the general (2 parameter family of) 2-step, $A$-stable
linear multi-step method is performed in the 1979 book Girault and Raviart
\cite{MR548867} but
there is no analogous stability or convergence analysis for the important case
of \textit{variable} timesteps. As an example of the challenges involved in
variable steps, BDF2 (a popular member of that $A$-stable family) loses
$A$-stability for increasing time steps, allowing non-physical energy growth. 
The instability is weak since $0$-stability is preserved for smoothly varying
timesteps, Boutelje and Hill
 \cite{Boutelje:2010:NonautonomousStabilityOfLinearMultistepMethod},
{S{\"o}derlind}, {Fekete} and {Farag{\'o}}
 \cite{2018arXiv180404553S}.
Similarly, the (2-leg) trapezoidal method can exhibit energy growth, when used
with variable steps
(Dahlquist, Liniger and Nevanlinna \cite{MR714701}, page 1073).
Liniger \cite{L83}\ presents a 2-step method that is non-autonomous $A$-stable
(applied to $y' = \lambda(t)y$). Dahlquist, Liniger and
Nevanlinna \cite{MR714701} give one that is $G$-stable (nonlinearly, energetically
stable, e.g., Dahlquist \cite{MR519054,dahlquist_tritaNA76218,MR520750}, Hairer,
N{\o }rsett and Wanner  \cite{MR1227985}) for
\textit{any} sequence of increasing or decreasing time-steps. Herein we give
an analysis of this method of Dahlquist, Liniger and Nevanlinna \cite{MR714701}
(the DLN method henceforth) for the Navier-Stokes Equations (NSE) with variable
timesteps.

Let $\Omega $ be the flow domain in ${\mathbb{R}}^{d}$ ($d=2\ \text{or\ }3$%
). The fluid velocity is denoted $u\left( x,t\right) $, pressure $p\left(
x,t\right) $ and body force $f\left( x,t\right) $. We analyze the variable
step, DLN time discretization for the NSE
\begin{align*}
& u_{t}+u\cdot \nabla {u}-\nu \Delta {u}+\nabla p=f,\ \ x\in \Omega ,\ \
0<t\leq T, \\
& \nabla \cdot u=0,\ \ x\in \Omega \ \ \text{for}\ \ 0<t\leq T,\ \
u(x,0)=u_{0}(x),\ \ x\in \Omega , \\
& u=0\ \ \text{on}\ \ \partial {\Omega },\ \ \int_{\Omega }p\,dx=0\ \ \text{%
for}\ \ 0<t\leq T.
\end{align*}%
Section \ref{section:variableDLN} recalls the DLN method. Applied to the
NSE, it takes the form
\begin{gather*}
\left( \frac{{\alpha _{2}}{u_{n+1}^{h}}+{\alpha _{1}}{u_{n}^{h}}+{\alpha _{0}%
}{u_{n-1}^{h}}}{\widehat{k}_{n}},v^{h}\right) +\nu (\nabla {u_{n,\ast }^{h}}%
,\nabla {v^{h}})
+b^{\ast }(u_{n,\ast }^{h},u_{n,\ast }^{h},{v^{h}})\\
-(p_{n,\ast }^{h},\nabla
\cdot {v^{h}})=(f(t_{n,\ast }),v^{h}), \\
(\nabla \cdot {u_{n+1}^{h}},q^{h})=0\text{, \ where }u_{n,\ast }=\sum_{\ell
=0}^{2}{\beta _{\ell }^{(n)}}u_{n-1+\ell }.
\end{gather*}%
Here  $\widehat{k}_{n}$ is a similar average of the variable time steps $%
k_{n-1}$ and $k_{n}$, and the multi-step method's coefficients ${\alpha _{2}}%
,{\alpha _{1}},{\alpha _{0},\beta }{_{2},\beta _{1},\beta _{0}}$ are given
in Section \ref{section:variableDLN}. The DLN method is a one-parameter
family (with parameter denoted $\theta $) and A-stable. Thus the constant
time step case (not considered herein) is a subset of the analysis in  Girault
and Raviart \cite%
{MR548867}. Section \ref{section:variableDLN} also presents its critical
property of variable step stability of $G$-stability with the \textit{%
G-matrix independent of the time step ratio}. Notations and preliminaries
are presented in Section \ref{section:notations}. Section \ref%
{section:stabilityDLN} gives a proof of variable timestep, unconditional,
long time, nonlinear stability of the one-leg DLN method for NSE. Let
$\Vert \cdot \Vert$ denote the $L^{2}$-norm. This analysis
shows that the natural kinetic energy, $\mathcal{E}(t_{n})$, and numerical
dissipation rate, $\mathcal{D}(t_{n})$, of the DLN approximation are%
\begin{eqnarray*}
\mathcal{E}(t_{n}) &=&{\frac{1}{4}}(1+\theta )\Vert u_{n}^{h}\Vert ^{2}+{%
\frac{1}{4}}(1-\theta )\Vert u_{n-1}^{h}\Vert ^{2},\quad \mbox{$\theta$=
method parameter}, \\
\mathcal{D}(t_{n}) &=&\frac{1}{\widehat{k}_{n}}\Big\|\sum_{\ell =0}^{2}{%
a_{\ell }^{(n)}}{u_{n-1+\ell }^{h}}\Big\|^{2},\text{ where the coefficients }%
a_{\ell }^{(n)}\mbox{ are given in \eqref{eq:aicoeff}}.
\end{eqnarray*}

Section \ref{section:erroranalysis} provides the variable step error
analysis. The DLN method is proven second-order for any sequence of time
steps. Numerical tests are presented in Section \ref{section:numericaltests}.
The first example confirms the theoretical prediction of second order
accuracy. The second test shows that DLN has stability advantages over BDF2
for variable timesteps. There is a recent idea of Capuano, Sanderse, Angelis
and Coppola \cite{Capuano2017AMT} to adapt the
time step to control the ratio of numerical to physical dissipation. Rather
than test a standard approach to error estimation and adaptivity, we also
test this idea in Section \ref{section:numericaltests}.

\subsection{Related work}

\label{subsection:relatedwork}

The number of papers studying timestepping methods
for flow problems is very large. The general (2 parameter) 2-step $A$-stable
method was analyzed for the NSE for \textit{constant} time steps in
Girault and Raviart \cite{MR548867}, and developed further by Jiang, Mohebujjaman
and Rebholz \cite{JIANG2016388}. Time adaptive discretizations of the NSE have
been limited by the Dahlquist barrier, storage limitations and the cognitive
complexity of extending to the NSE many of the standard methods for systems
of ordinary differential equations. One early and important work is that of
Kay, Gresho, Griffiths and Silvester \cite{MR2599769}. It presents an adaptive
algorithm based on the trapezoid
scheme / linearized midpoint rule (with error estimation done using an
explicit AB2 type method) that is memory and computation efficient. It is
well known for systems of ODEs that variable step, variable order (VSVO)
methods are the ones of choice. These have only been considered for the NSE
in three recent works, Hay, Etienne, Pelletier and Garon \cite{MR3327978}
(based on the BDF family), Decaria, Guzel and Li \cite{DLL18}, Decaria and
Zhao \cite{DLZ18} (based on time filters). The methods based on time filters are
promising but relatively unexplored. For example, their variable step
$G$-stability is unknown.

%
%

\section{The variable step DLN\ method}

\label{section:variableDLN}

The DLN method is a 1-parameter ($0\leq \theta \leq 1$) family of $A$-stable,
2-step, $G$-stable \ methods. If $\theta =1$ it reduces to the one-step,
one-leg trapezoid
(midpoint) scheme. Its key property is that\textit{\ the G-stability matrix
depends on the parameter }$\theta $\textit{\ but not on the timestep ratio}
in Lemma \ref{theorem:GstabilityDLN} below.
Let $y:[0,T]\rightarrow \mathbb{R}^{d},f:\mathbb{R}\times {\mathbb{R}}%
^{d}\rightarrow {\mathbb{R}}^{d}$. Consider the initial value problem
\begin{equation*}
y^{\prime }(t)=f(t,y(t)),\quad y\left( 0\right) =y_{0}.
\end{equation*}%
Let partition $P$ on $[0,T]$ be $\{t_{n}\}_{n=0}^{M}$ ($M\in \mathbb{N}$)
where
\begin{equation*}
0=t_{0}<t_{1}<\cdots <t_{M-1}<t_{M}=T.
\end{equation*}%
We recall the following notation from Dahlquist, Liniger and Nevanlinna
\cite{MR714701} for the local step
size $k_{n}$, the stepsize variability $\varepsilon _{n}\in (-1, 1)$:
\begin{equation*}
k_{n}=t_{n+1}-t_{n},\qquad \varepsilon _{n}=\frac{k_{n}-k_{n-1}}{%
k_{n}+k_{n-1}},
\end{equation*}%
and the coefficients $\{\alpha _{\ell },\beta _{\ell }\}_{\ell =0:2}$ are
\begin{align}  \label{eq:DLNcoeffs-explicit}
\left(
\begin{array}{ll}
\alpha _{2} & \beta _{2}^{(n)}\vspace{0.2cm} \\
\alpha _{1} & \beta _{1}^{(n)}\vspace{0.2cm} \\
\alpha _{0} & \beta _{0}^{(n)}%
\end{array}%
\right) =\left(
\begin{array}{lll}
\frac{1}{2}(\theta +1) &  & \frac{1}{4}\Big(1+\frac{1-{\theta }^{2}}{(1+{%
\varepsilon _{n}}{\theta })^{2}}+{\varepsilon _{n}}^{2}\frac{\theta (1-{%
\theta }^{2})}{(1+{\varepsilon _{n}}{\theta })^{2}}+\theta \Big)\vspace{0.2cm%
} \\
-\theta  &  & \frac{1}{2}\Big(1-\frac{1-{\theta }^{2}}{(1+{\varepsilon _{n}}{%
\theta })^{2}}\Big)\vspace{0.2cm} \\
\frac{1}{2}(\theta -1) &  & \frac{1}{4}\Big(1+\frac{1-{\theta }^{2}}{(1+{%
\varepsilon _{n}}{\theta })^{2}}-{\varepsilon _{n}}^{2}\frac{\theta (1-{%
\theta }^{2})}{(1+{\varepsilon _{n}}{\theta })^{2}}-\theta \Big)%
\end{array}%
\right) .
\end{align}%
%
For constant time steps, the DLN stability region boundary with
$\theta = \frac{1}{2}$ and that of BDF2
for comparison plotted by the root locus are given in Figure \ref{fig0}.
\begin{figure}[ptbh]
\centering
\par
{
		\begin{minipage}[t]{0.45\linewidth}
			\centering
			\includegraphics[width=2.3in,height=1.8in]{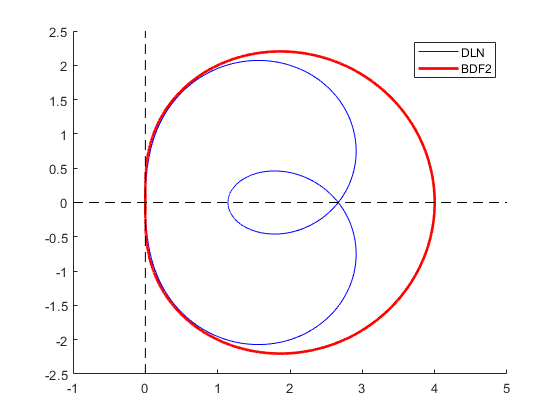}\\
			\vspace{0.02cm}
					\end{minipage}	} \centering
\vspace{-0.2cm}
\caption{Boundaries of Stability Region for constant DLN and BDF2.}
\label{fig0}
\end{figure}
We also recall the definitions of the DLN's averaged timestep $\widehat{k}_{n}$:
\begin{equation*}
\widehat{k}_{n}={\alpha _{2}}k_{n}-{\alpha _{0}}k_{n-1}=\frac{1}{2}(1+\theta
)k_{n}+\frac{1}{2}(1-\theta )k_{n-1}=\theta \frac{k_{n}-k_{n-1}}{2}+\frac{%
k_{n}+k_{n-1}}{2},
\end{equation*}%
%
and the coefficients $a_{\ell }^{(n)}$:
\begin{align}   \label{eq:aicoeff}
a_{1}^{(n)}=-\frac{\sqrt{\theta \left( 1-{\theta }^{2}\right) }}{\sqrt{2}%
(1+\varepsilon _{n}\theta )},\quad a_{2}^{(n)}=-\frac{1-\varepsilon _{n}}{2}%
a_{1}^{(n)},\quad a_{0}^{(n)}=-\frac{1+\varepsilon _{n}}{2}a_{1}^{(n)},
\end{align}%
which are used in the expression of the numerical dissipation.

The $\alpha_\ell$-coefficients do not depend on the time-step ratio.
The $\beta_\ell$- and $a_\ell$-coefficients depend on the time-step
ratios through the variability coefficients $\varepsilon_n$. %
\newline

The one-leg DLN method is then
\begin{align}  \label{eq:1legDLN}
\sum_{\ell =0}^{2}{\alpha _{\ell }}y_{n-1+\ell }=\widehat{k}_{n}{f}\Big(%
\sum_{\ell =0}^{2}{\beta _{\ell }^{(n)}}t_{n-1+\ell },\sum_{\ell =0}^{2}{%
\beta _{\ell }^{(n)}}y_{n-1+\ell }\Big). \tag{DLN}
\end{align}%
Let $\left\Vert \cdot \right\Vert $ and $\left( \cdot ,\cdot \right) _{{%
\mathbb{R}}^{d}}$ denote in this section the usual norm and inner
product on Euclidean space ${\mathbb{R}}^{d}.$

\begin{definition}
For $\theta \in [0,1]$, define the symmetric semi-positive definite $G(\theta)$
matrix
\begin{align}  \label{eq:Gmatrix}
G(\theta) = \left(%
\begin{array}{ll}
\frac{1}{4} (1+\theta) \mathbb{I}_d & 0 \vspace{.2cm} \\
0 & \frac{1}{4} (1-\theta) \mathbb{I}_d%
\end{array}%
\right),
\end{align}
with the corresponding G-norm 
\begin{align}  \label{eq:G-norm}
\begin{Vmatrix}
u \\
v%
\end{Vmatrix}
_{G(\theta)}^{2}:={\frac{1}{4}}(1+{\theta})\left\Vert u\right\Vert ^{2}+{%
\frac{1}{4}}(1-{\theta})\left\Vert v\right\Vert ^{2}\ \ \ \text{for \ }u,v\in%
{\mathbb{R}}^{d}.
\end{align}
\end{definition}

%
Recall the following result, from  Dahlquist, Liniger and Nevanlinna
\cite{MR714701}, related to the
$G$-stability of the DLN method \eqref{eq:1legDLN}, which will be used in
proving main theorems herein.

\begin{lemma}
\label{theorem:GstabilityDLN} Let $0\leq \theta \leq 1$. The variable step,
one-leg DLN method (\ref{eq:1legDLN}) is G-stable, i.e. for any $%
n=1,2,\cdots M-1$, with $a_{\ell }^{(n)}\ (\ell =0,1,2)$ given above %
\eqref{eq:aicoeff}, we have
\begin{align}   \label{eq:Gstab}
\Big(\sum_{\ell =0}^{2}{\alpha _{\ell }}y_{n-1+\ell },\sum_{\ell =0}^{2}{%
\beta _{\ell }^{(n)}}y_{n-1+\ell }\Big)_{{\mathbb{R}}^{d}}=%
\begin{Vmatrix}
{y_{n+1}} \\
{y_{n}}%
\end{Vmatrix}%
_{G(\theta )}^{2}-%
\begin{Vmatrix}
{y_{n}} \\
{y_{n-1}}%
\end{Vmatrix}%
_{G(\theta )}^{2}+\Big\|\sum_{\ell =0}^{2}{a_{\ell }^{(n)}}y_{n-1+\ell }%
\Big\|^{2}.
\end{align}
\end{lemma}

%

\begin{proof}
The proof (implicit in Dahlquist, Liniger and Nevanlinna \cite{MR714701}) is
an algebraic calculation. \ \ \
\end{proof}

\section{Notation and Preliminaries}

\label{section:notations}

Let $\Omega$ be any domain in ${\mathbb{R}}^{d}$ ($d=2\ \text{or\ }3$). For $%
1\leq p<\infty$, $\Vert\cdot\Vert_{L^{p}}$ and $\Vert\cdot\Vert_{W_{p}^{k}}$
are norms on function spaces $L^{p}\left( \Omega\right) $ and $%
W_{p}^{k}\left( \Omega\right) $ respectively. There is a special case: if $%
p=2$, we denote $\Vert\cdot\Vert$ be $L^{2}$-norm
with inner product $( \cdot,\cdot )$. $H^{k}\left(
\Omega\right) $ is the Sobolev space $W_{2}^{k}\left( \Omega\right) $ with
norm $\Vert\cdot\Vert_{k}$. \noindent The velocity and pressure $\left(
u,p\right) $ are in the spaces $\left( X,Q\right) $
given by
\begin{align*}
& X = \Big\{ v:\Omega\rightarrow{\mathbb{R}}^{d}:v\in L^{2}\left(
\Omega\right) ,\nabla{v}\in L^{2}\left( \Omega\right) \ \text{and \ }v=0\
\text{on \ }\partial{\Omega}\Big\} , \\
& Q = \Big\{ q:\Omega\rightarrow{\mathbb{R}}:v\in L^{2}\left( \Omega\right)
\ \text{and \ }\int_{\Omega}q \, dx =0 \Big\} .
\end{align*}
The spaces of divergence free functions is denoted
\begin{align*}
V =\Big\{ v\in X:\left( \nabla\cdot v,q\right) =0,\ \forall q\in Q \Big\}.
\end{align*}
The space $X^{\ast}$ and $V^{\ast}$ are the dual space of $X$ and $V$ with
norms given by
\begin{align*}
\left\Vert f\right\Vert _{-1}:=\sup_{0\neq v\in X}\frac{\left( f,v\right) }{%
\left\Vert \nabla{v}\right\Vert }\ ,\ \ \ \left\Vert f\right\Vert _{\ast
}:=\sup_{0\neq v\in V}\frac{\left( f,v\right) }{\left\Vert \nabla {v}%
\right\Vert },
\end{align*}
respectively. For functions $v\left( x,t\right) $ and $1\leq p<\infty$,
define
\begin{align*}
\Vert v\Vert_{\infty,k} := \limfunc{ess\,sup}_{0<t<T}\Vert v\left( t,\cdot
\right) \Vert_{k}\ \ \ \text{and \ }\Vert v\Vert_{p,k}:=\left(
\int_{0}^{T}\Vert v\left( t,\cdot\right) \Vert_{k}^{p}dt\right) ^{1/p}\ .
\end{align*}
\noindent For $u,v,w\in X$, define the explicitly skew symmetrized trilinear
form
\begin{align*}
b^{\ast}\left( u,v,w\right) :=\frac{1}{2}\left( u\cdot\nabla{v},w\right) -%
\frac{1}{2}\left( u\cdot\nabla{w},v\right) .
\end{align*}
\noindent$b^{\ast}\left( u,v,w\right) $ satisfies the bound,
\cite[p.11 Lemma\ 3]{LAYTON2008},
\begin{align} \label{eq:conti-triform}
b^{\ast}\left( u,v,w\right) \leq C\left( \Omega\right) \Vert\nabla{u}%
\Vert\Vert\nabla{v}\Vert\Vert\nabla{w}\Vert,  \notag \\
b^{\ast}\left( u,v,w\right) \leq C\left( \Omega\right) \Vert u\Vert
^{1/2}\Vert\nabla{u}\Vert^{1/2}\Vert\nabla{v}\Vert\Vert\nabla{w}\Vert.
\end{align}
We recall the following standard lemma for $b^{\ast}$

\begin{lemma}
\label{lemma:skew-symmetry} For any $u,v,w \in X$
\begin{align}  \label{eq:skew-symmetry}
b^{*} \left( u, v, v \right) =0 ,
\end{align}
and
\begin{align}
b^{*} \left( u, v, w \right) = \left( u \cdot\nabla{v}, w \right) ,
\end{align}
for all $u \in V$ and $v,w \in X$.
\end{lemma}

\begin{proof}
By definition of $b^{*}$, we have $b^{*} \left( u, v, v \right) =0, \
\forall u,v \in X$. For second part, integrate by parts then use
$\nabla \cdot u =0$ and $u|_{\partial{\Omega}} = 0$.\ \ \
\end{proof} 

\noindent We base our analysis on the finite element method (FEM) for the
spatial discretization. The approximate solutions for the velocity and pressure are
in the finite element spaces, based on an edge to edge triangulation $\Omega $
(with maximum triangle diameter $h$ ) denoted by
\begin{equation*}
X_{h}\subset X,\ \ \ Q_{h}\subset Q.
\end{equation*}%
We assume that $X_{h}$ and $Q_{h}$ satisfy the usual discrete inf-sup
condition ($LBB^{h}$ condition). The Taylor-Hood elements, which satisfy the
condition, are used in the numerical tests. The discretely divergence-free
subspace of $X_{h}$ is
\begin{equation*}
V_{h}:=\left\{ v_{h}\in X_{h}:\left( \nabla \cdot v_{h},q_{h}\right) =0,\
\forall q_{h}\in Q_{h}\right\} .
\end{equation*}%
We also need the following interpolation error estimate for the velocity $u$ and pressure
$p$: for $k,s \in \mathbb{N}$,
\begin{align} \label{eq:interpolation-Error}
\Vert u-I^{h}u\Vert_{r}\leq Ch^{k+1-r}\Vert u\Vert_{k+1},\ \ \ u\in {%
H^{k+1}\left( \Omega\right) }^{d}, \ \ \ 0 \leq r \leq k  \notag \\
\Vert p-I^{h}p\Vert_{r}\leq Ch^{s+1-r}\Vert p\Vert_{s+1},\ \ \ p\in
H^{s+1}\left( \Omega\right), \ \ \ 0 \leq r \leq s
\end{align}
where $I^{h}u$ and $I^{h}p$ are the $L^{2}$ projection of $u$ and $p$ onto $X^{h}$ and $%
Q^{h}$ respectively, see e.g. Brenner and Scott \cite{MR1278258}.

Let $[0,T]$ be a time interval, $P_{0}=\{t_{n}\}_{n=0}^{M}$ a partition on $%
[0,T]$, and $\{k_{n}\}_{n=0}^{M-1}$ denote the set of time-step sizes.

\begin{definition}
For any given sequence $\{z_{n}\}_{n\geq 1}$, we denote by
\begin{equation*}
z_{n,\ast }=\sum_{\ell =0}^{2}{\beta _{\ell }^{(n)}}z_{n-1+\ell }
\end{equation*}
the convex combination of the three adjacent terms in the sequence.
\end{definition}

As examples, $\{t_{n,\ast }\}$ is the set of time-values 
and $u_{n,\ast }$ are the implicit values where the equation is evaluated%
\begin{eqnarray*}
t_{n,\ast } &=&\beta _{2}^{(n)}t_{n+1}+\beta _{1}^{(n)}t_{n}+\beta
_{0}^{(n)}t_{n-1}, \\
u_{n,\ast } &=&\beta _{2}^{(n)}u_{n+1}+\beta _{1}^{(n)}u_{n}+\beta
_{0}^{(n)}u_{n-1}.
\end{eqnarray*}%
The variational formulation of the one-leg DLN method for the NSE is as
follows. With the DLN coefficients 
\eqref{eq:DLNcoeffs-explicit}, given $u_{n}^{h},u_{n-1}^{h}\in X_{h}$ and $%
p_{n}^{h},p_{n-1}^{h}\in Q_{h}$, find $u_{n+1}^{h}$ and $p_{n+1}^{h}$
satisfying
\begin{gather}
\Big(\frac{{\alpha _{2}}{u_{n+1}^{h}}+{\alpha _{1}}{u_{n}^{h}}+{\alpha _{0}%
}{u_{n-1}^{h}}}{\widehat{k}_{n}},v^{h}\Big)+\nu (\nabla {u_{n,\ast }^{h}}%
,\nabla {v^{h}})+b^{\ast }(u_{n,\ast }^{h},u_{n,\ast }^{h},{v^{h}}%
) \notag  -(p_{n,\ast }^{h},\nabla \cdot {v^{h}})
\\
=(f(t_{n,\ast }),v^{h})\qquad \forall v^{h}\in X^{h},  \label{eq:DLN-NSE}\\
(\nabla \cdot {u_{n+1}^{h}},q^{h})=0\qquad \forall q^{h}\in Q^{h}.  \notag
\end{gather}%
\noindent Under the discrete inf-sup condition, \eqref{eq:DLN-NSE} 
is equivalent to
\begin{align}  \label{eq:DLN-NSEinfsup}
\Big(\frac{{\alpha _{2}}{u_{n+1}^{h}}+{\alpha _{1}}{u_{n}^{h}}+{\alpha _{0}}{%
u_{n-1}^{h}}}{\widehat{k}_{n}},v^{h}\Big)+\nu (\nabla {u_{n,\ast }^{h}}%
,\nabla {v^{h}})+b^{\ast }(u_{n,\ast }^{h},u_{n,\ast }^{h},{v^{h}}%
)=(f_{n,\ast },v^{h}),\quad \forall v^{h}\in V^{h}.
\end{align}

\noindent Furthermore, we need the following variable timestep, discrete
Gronwall inequality (see Heywood and Rannacher \cite{Hey2} for the proof).

\begin{lemma}
\label{lemma:discrete-Gronwall}
Let $a_{n},b_{n},c_{n},d_{n},k,B$ be nonnegative numbers, for integers $%
n\geq 0$ such that
\begin{equation*}
a_{\ell}+ k \sum_{n=0}^{\ell}b_{n}\leq
k \sum_{n=0}^{\ell}d_{n}a_{n}+k\sum_{n=0}^{\ell}c_{n}+B\ \ \ \text{for
\ }\ell \geq 0.
\end{equation*}%
Suppose that $kd_{n}<1$ for all $n$, then
\begin{equation*}
a_{\ell}+k\sum_{n=0}^{\ell}b_{n}\leq \exp { \left(\sum_{n=0}^{\ell}
\frac{kd_{n}}{1 - kd_{n}} \right) }\left(
k\sum_{n=0}^{\ell}c_{n}+B\right) \ \ \ \text{for \ }\ell\geq 0.
\end{equation*}
\end{lemma}

\section{Stability of DLN for the NSE}

\label{section:stabilityDLN}

In this section, we prove the unconditional, long time, variable timestep
energy-stability of \eqref{eq:DLN-NSE}, using the $G$-stability property %
\eqref{eq:Gstab} of the method.

\begin{theorem}[Unconditional, Long Time Stability] \label{theorem:DLN-NSEStab}
The one-leg DLN method by \eqref{eq:DLN-NSE} or \eqref{eq:DLN-NSEinfsup} is
unconditionally, long-time stable: for any integer $M>1$,
\begin{gather}
{\frac{1}{4}}(1+\theta )\Vert u_{M}^{h}\Vert ^{2}+{\frac{1}{4}}(1-\theta
)\Vert u_{M-1}^{h}\Vert ^{2}+\sum_{n=1}^{M-1}\Bigg\|\sum_{\ell =0}^{2}{%
a_{\ell }^{(n)}}{u_{n-1+\ell }^{h}}\Bigg\|^{2}+\frac{\nu }{2}\sum_{n=1}^{M-1}%
{\widehat{k}_{n}}\Vert \nabla {u_{n,\ast }^{h}}\Vert ^{2}  \notag \\
\leq \frac{1}{2{\nu }}\sum_{n=1}^{M-1}\widehat{k}_{n}\Vert f(t_{n,\ast
})\Vert _{\ast }^{2}+\frac{1}{4}(1+\theta )\Vert u_{1}^{h}\Vert ^{2}+\frac{1%
}{4}(1-\theta )\Vert u_{0}^{h}\Vert ^{2},  \notag
\end{gather}%
where $a_{i}^{(n)},i=0,1,2$, given previously by \eqref{eq:aicoeff}, are
\begin{equation*}
a_{1}^{(n)}=-\frac{\sqrt{\theta \left( 1-{\theta }^{2}\right) }}{\sqrt{2}%
\left( 1+\varepsilon _{n}\theta \right) },\ a_{2}^{(n)}=-\frac{1-\varepsilon
_{n}}{2}a_{1}^{(n)},\ a_{0}^{(n)}=-\frac{1+\varepsilon _{n}}{2}a_{1}^{(n)}.
\end{equation*}
\end{theorem}

\begin{proof}
For $n=1,\cdots, M-1$, set $v^{h}=u_{n,\ast}^{h}$ in \eqref{eq:DLN-NSEinfsup}.
Then, using the skew-symmetry relation \eqref{eq:skew-symmetry}
\begin{confidential}
\color{cyan}
\begin{align*}
{\frac{1}{\widehat{k}_n }} \Big( \sum_{l=0}^{2}{\alpha_{l}}{u_{n-1+l}^{h}} , u_{n,\ast}^{h} \Big)
+ \nu \| \nabla{u_{n,\ast}^{h}} \|^{2}
= ( f  (  t_{n,\ast} )  , \nabla{u_{n,\ast}^{h}} )  .
\end{align*}
Using definition of $\Vert\cdot\Vert_{\ast}$ norm and Young's inequality
\begin{align*}
\left(  f\left(  t_{n,\ast}\right)  ,\nabla{u_{n,\ast}^{h}}\right)  \leq\Vert
f\left(  t_{n,\ast}\right)  \Vert_{\ast}\Vert\nabla{u_{n,\ast}^{h}}\Vert
\leq{\frac{\nu}{2}}\Vert\nabla{u_{n,\ast}^{h}}\Vert^{2}+{\frac{1}{2{\nu}}%
}\Vert f\left(  t_{n,\ast}\right)  \Vert_{\ast}^{2}\ .
\end{align*}
Thus for $n=1,2,\cdots M-1$
\\
\normalcolor
\end{confidential}
and the Cauchy-Schwarz inequality, we obtain
\begin{align*}
\Big( \sum_{\ell=0}^{2}{\alpha_{\ell}}{u_{n-1+\ell}^{h}}\ ,u_{n,\ast}^{h} %
\Big) + \frac{\nu}{2} \widehat{k}_n \| \nabla{u_{n,\ast}^{h}} \|^{2} \leq
\frac{1}{2{\nu}} \widehat{k}_n \| f ( t_{n,\ast} ) \|_{\ast}^{2}.
\end{align*}
The $G$-stability relation \eqref{eq:Gstab} implies%
\begin{align*}
\begin{Vmatrix}
{u_{n+1}^{h}} \\
{u_{n}^{h}}%
\end{Vmatrix}%
_{G(\theta)}^{2} -
\begin{Vmatrix}
{u_{n}^{h}} \\
{u_{n-1}^{h}}%
\end{Vmatrix}%
_{G(\theta)}^{2} + \Bigg\| \sum_{\ell=0}^{2} a_{\ell}^{(n)} u_{n-1+\ell}^{h} %
\Bigg\| ^{2} + \frac{\nu}{2} \widehat{k}_n \| \nabla{u_{n,\ast}^{h}} \|^2
\leq \frac{1}{2{\nu}} \widehat{k}_n \| f ( t_{n,\ast}) \|_{\ast}^{2}.
\end{align*}
Summation over $n$ from $1$ to $M-1$, and the definition \eqref{eq:G-norm}
yields
\begin{confidential}
\color{cyan}

\begin{align*}%
\begin{Vmatrix}
{u_{M}^{h}}\\
{u_{M-1}^{h}}%
\end{Vmatrix}
_{G(\theta)}^{2}-%
\begin{Vmatrix}
{u_{1}^{h}}\\
{u_{0}^{h}}%
\end{Vmatrix}
_{G(\theta)}^{2}+\sum_{n=1}^{M-1}\left\Vert \sum_{l=0}^{2}{a_{l}^{(n)}%
}{u_{n-1+l}^{h}}\right\Vert ^{2}\\
+\sum_{n=1}^{M-1}{\frac{\nu}{2}}{\widehat{k}_n }\left\Vert \nabla{u_{n,\ast}^{h}%
}\right\Vert ^{2}\leq\sum_{n=1}^{M-1}{\frac{1}{2{\nu}}}{\widehat{k}_n }\Vert
f\left(  t_{n,\ast}\right)  \Vert_{\ast}^{2}\ .
\end{align*}
\noindent The above inequality results in the proof.
\\
\normalcolor
\end{confidential}
the conclusion.
\end{proof}

The above stability result identifies the DLN method's kinetic energy and
numerical energy dissipation rates: %
\begin{align*}
\begin{array}{l}
\mathcal{E}_n= \displaystyle \frac{1}{4} (1+\theta) \| u_{n}^{h} \|^2 +
\frac{1}{4} (1-\theta) \| u_{n-1}^{h} \|^2 , \vspace{.2cm} \ \
\mathcal{D}_n = \displaystyle \frac{1}{\widehat{k}_n} \Big\| %
\sum_{\ell=0}^{2} a_{\ell}^{(n)} u_{n-1+\ell}^{h} \Big \|^{2}.%
\end{array}%
\end{align*}
%

\section{Variable Time-step Error Analysis}

\label{section:erroranalysis}

In this section, we analyze the error in the approximate solutions by the one-leg
DLN method for variable time steps. The discrete time error analysis
requires norms that are discrete time analogues of the norms used in the
continuous time case. As before, let $[0,T]$ denote the whole time interval, $%
P_{0}=\left\{ t_{n}\right\} _{n=0}^{M}$ be a partition on $[0,T]$ and $%
\left\{ k_{n}\right\} _{n=0}^{M-1}$ be the set of time-step sizes. For a function $%
v\left( x,t\right) $ and $1\leq p<\infty$, we define
\begin{gather*}
\left\Vert \left\vert v\right\vert \right\Vert _{\infty,k} = \max_{0\leq
n\leq M} \Vert v_{n} \Vert _{k}, \ \Vert |v| \Vert_{p,k}^{P_{0},L} = %
\Big( \sum_{n=0}^{M-1} k_{n}\Vert v_{n}\Vert_{k}^{p} \Big) ^{1/p},\ \Vert|
 v |\Vert_{p,k}^{P_{0},R} = \Big( \sum_{n=1}^{M}k_{n-1}\Vert v_{n}\Vert_{k}^{p} \Big) %
^{1/p}.
\end{gather*}
In the above definitions, the last two terms are forms of Riemann sums in which the
function $v$ is evaluated at the left endpoint or right endpoint of each
small time interval $[t_{n},t_{n+1}]$. $P$ is the given partition on $[0,T]$
and $L,R$ means that the sum involves the value of the function at the left endpoint
or right endpoint of each time interval $[t_{n},t_{n+1}]$ respectively.
\newline

\noindent Then we define two new partitions related to partition $P_{0}$: If
$M$ is odd
\begin{align*}
& P_{1} := \left\{ s_{\ell} : 0 = s_{0} < s_{1} < \cdots< s_{\frac{M+1}{2}} = T
\ \text{and \ } s_{\ell} = t_{2\ell} \ \text{for \ } 1 \leq \ell \leq\frac{M-1}{2}
\right\} , \\
& P_{2} := \left\{ s_{\ell} : 0 = s_{0} < s_{1} < \cdots< s_{\frac{M+1}{2}} = T
\ \text{and \ } s_{\ell} = t_{2\ell-1} \ \text{for \ } 1 \leq \ell \leq\frac{M-1}{2}
\right\} ,
\end{align*}
and if $M$ is even
\begin{align*}
& P_{1} := \left\{ s_{\ell} : 0 = s_{0} < s_{1} < \cdots< s_{\frac{M}{2}} = T \
\text{and \ } s_{\ell} = t_{2\ell} \ \text{for \ } 1 \leq \ell \leq\frac{M}{2} - 1
\right\} , \\
& P_{2} := \left\{ s_{\ell} : 0 = s_{0} < s_{1} < \cdots< s_{\frac{M}{2}+1} = T
\ \text{and \ } s_{\ell} = t_{2\ell-1} \ \text{for \ } 1 \leq \ell \leq\frac{M}{2}
\right\} .
\end{align*}
\noindent Based on the partitions above, define
\begin{align*}
& \left\Vert \left| v \right| \right\Vert _{p,k} := \left( \sum_{\ell=0}^{2}
\left( \left\Vert \left| v \right| \right\Vert _{p,k}^{P_{\ell},R} \right) ^{p}
+ \sum_{\ell=0}^{2} \left( \left\Vert \left| v \right| \right\Vert
_{p,k}^{P_{\ell},L} \right) ^{p} \right) ^{1/p} .
\end{align*}
\noindent Furthermore given the partitions $\left\{ P_{\ell} \right\}_{\ell=0}^{2}$
above, define the new partitions $\widetilde{P}_{\ell} $ ($\ell =1,2$): if $M$ is
odd,
\begin{align*}
& \widetilde{P}_{1} := \left\{ s_{\ell} : 0 = s_{0} < s_{1} < \cdots< s_{\frac{%
M-1}{2}} = t_{M-1} \ \text{and \ } s_{\ell} = t_{2\ell} \ \text{for \ } 1 \leq \ell
\leq\frac{M-3}{2} \right\} , \\
& \widetilde{P}_{2} := \left\{ s_{\ell} : t_{1} = s_{0} < s_{1} < \cdots< s_{%
\frac{M-1}{2}} = T \ \text{and \ } s_{\ell} = t_{2\ell+1} \ \text{for \ } 1 \leq \ell
\leq\frac{M-3}{2} \right\} ,
\end{align*}
if $M$ is even,
\begin{align*}
& \widetilde{P}_{1} := \left\{ s_{\ell} : 0 = s_{0} < s_{1} < \cdots< s_{\frac {%
M}{2}} = T \ \text{and \ } s_{\ell} = t_{2\ell} \ \text{for \ } 1 \leq \ell \leq
\frac{M}{2} - 1 \right\} , \\
& \widetilde{P}_{2} := \left\{ s_{\ell} : t_{1} = s_{0} < s_{1} < \cdots< s_{%
\frac{M}{2}-1} = t_{M-1} \ \text{and \ } s_{\ell} = t_{2\ell+1} \ \text{for \ } 1
\leq \ell \leq\frac{M}{2} - 2 \right\} .
\end{align*}
\noindent For $\widetilde{P}_{1}$, we have $t_{2\ell-1} \in[t_{2\ell-2},t_{2\ell}] =
[s_{\ell-1},s_{\ell}]$ and let $\bar{s}_{\ell} := t_{2\ell-1,\ast}$. Similarly for
$\widetilde{P}_{2}$, $t_{2\ell} \in[t_{2\ell-1},t_{2\ell+1}] = [s_{\ell-1},s_{\ell}]$,
$\bar{s}_{\ell} := t_{2\ell, \ast}$. For the function $v(x,t)$ above, define
\begin{align*}
\left\Vert \left| v_{\ast} \right| \right\Vert _{p,k}^{\widetilde{P}_{i}} := %
\Bigg( \sum_{\ell=1}^{\# \widetilde{P}_{i} - 1} \left( s_{\ell} - s_{\ell-1} \right)
\left\Vert v \left( \bar{s}_{\ell} \right) \right\Vert _{k}^{p} \Bigg) ^{1/p},
\ \ \ i = 1,2,
\end{align*}
where $\# \widetilde{P}_{i}$ is number of set $\widetilde{P}_{i}$, and
\begin{align*}
\| | v_{\ast} | \| _{p,k} := \Big( \big( \left\Vert \left| v_{*} \right|
\right\Vert _{p,k}^{\widetilde{P}_{1}} \big) ^{p} + \big( \left\Vert \left|
v_{*} \right| \right\Vert_{p,k}^{\widetilde{P}_{2}} \big) ^{p} \Big) ^{1/p} .
\end{align*}
\noindent Now we introduce the following lemma to be used often in error
analysis.

\begin{lemma}
\label{lemma:disnormEst}
Let $v$ be a continuous function on interval $[0,T] \times\Omega$ and $\left\{
P_{\ell} \right\} _{\ell=0}^{2}, \{ \widetilde{P}_{\ell} \}_{\ell=1}^{2}$ be
the partitions on $[0,T]$ same as stated above. Then for any $1 \leq p <
\infty$, we have
\begin{align*}
\sum_{n=1}^{M-1} \left( k_{n} + k_{n-1} \right) \sum_{\ell=0}^{2} \left\Vert
v_{n-1+\ell} \right\Vert _{p,k+1}^{p} \leq{\ \left\Vert \left| v \right|
\right\Vert _{p,k} }^{p} ,
\end{align*}
and
\begin{align*}
\sum_{n=1}^{M-1} \left( k_{n} + k_{n-1} \right) \left\Vert v \left( t_{n,*}
\right) \right\Vert _{k}^{p} \leq\left( \left\Vert \left| v_{*} \right|
\right\Vert _{p,k} \right) ^{p} .
\end{align*}
\end{lemma}

\noindent We can also define the discrete norm of functions with respect to the dual norm $%
\left\Vert \cdot\right\Vert _{*}$, and derive a related lemma similar
to Lemma \ref{lemma:disnormEst}. Moreover we need the following lemma dealing with consistency
error.

\begin{lemma}[consistency errors]
\label{lemma:consistError}
Let $u(t)$ be any continuous function on $[0,T]$. If $u_{tt} \in
L^{2}\left(\Omega \times(t_{n-1},t_{n}) \right)$, then
\begin{align*}
\left\Vert \sum_{\ell=0}^{2} \beta_{\ell}^{(n)}u(t_{n-1+\ell})
- u \left( t_{n,*} \right) \right\Vert^{2} \leq C \left(
k_{n} + k_{n-1} \right) ^{3} \int_{t_{n-1}}^{t_{n+1}} \left\Vert u_{tt}
\right\Vert ^{2} dt.
\end{align*}
For $\theta\in[0,1)$, if $u_{ttt} \in
L^{2}\left(\Omega \times(t_{n-1},t_{n}) \right)$, then
\begin{align*}
\left\Vert \frac{{\alpha_{2}}{u(t_{n+1})} + {\alpha_{1}}{u(t_{n})} + {\alpha_{0}}{%
u(t_{n-1})}}{\widehat{k}_n } - u_{t} \left( t_{n,*} \right) \right\Vert ^{2}
\leq C \left( \theta\right) \left( k_{n} + k_{n-1} \right) ^{3}
\int_{t_{n-1}}^{t_{n+1}} \Vert u_{ttt} \Vert^{2} dt .
\end{align*}
\end{lemma}

\begin{proof}
The proof for smooth functions is simply Taylor expansion with integral reminder
after expanding
function $u\left( t_{n+1} \right),u\left( t_{n-1} \right)$ and
$u \left( t_{n,*} \right) $ at $t_{n}$. For less smooth
functions it then follows by a density argument.
\end{proof} 

\noindent Now we introduce the main theorem about
error analysis under the following timestep condition:
\begin{gather} \label{eq:timestep-restriction}
C\left( \theta \right)\sum_{\ell=0}^{2} \left( {{\nu }^{-3}}\Vert
\nabla {u_{n-1+\ell,\ast }}\Vert ^{4}+1\right) \widehat{k}_{n-1+\ell}< 1
\end{gather}
for $2 \leq n \leq M-2$.

\begin{theorem}
\label{theorem:DLN-NSE-Error}
Let $(u(t),p(t))$ be a sufficiently smooth, strong solution of the NSE.
When applying one-leg DLN's algorithm \eqref{eq:DLN-NSE} or
\eqref{eq:DLN-NSEinfsup},  there is a constant $C>0$ such that
under timestep condition \eqref{eq:timestep-restriction}, the following error
estimates hold
\begin{equation*}
\Vert |u-u^{h}|\Vert _{\infty ,0}\leq Ch^{k+1}\Vert |u|\Vert _{\infty ,k+1}+F%
\Big(h,\max_{1\leq n\leq M-1}(k_{n}+k_{n-1})\Big),
\end{equation*}%
and
\begin{align*}
&\Big({\nu }\sum_{n=1}^{M-1}\widehat{k}_{n}\Vert \nabla (u(t_{n,\ast
})-u_{n,\ast }^{h})\Vert ^{2}\Big)^{\frac{1}{2}} \\
\leq &C\nu ^{\frac{1}{2}}\max_{1\leq n\leq M-1}\{(k_{n}+k_{n-1})^{2}\}\Vert
\nabla {u_{tt}}\Vert _{2,0}+C{\nu }^{\frac{1}{2}}h^{k}\Vert |u|\Vert
_{2,k+1}+F\Big(h,\max_{1\leq n\leq M-1}(k_{n}+k_{n-1})\Big),
\end{align*}%
where
\begin{gather*}
F\Big(h,\max_{1\leq n\leq M-1}(k_{n}+k_{n-1})\Big)=C{\nu }^{\frac{1}{2}%
}h^{k}\Vert |u|\Vert _{2,k+1} \\
+C{\nu }^{-\frac{1}{2}}h^{k+\frac{1}{2}}\Big(\Vert |u|\Vert
_{4,k+1}^{2}+\Vert |\nabla {u}|\Vert _{4,0}^{2}\Big)+C{{\nu }^{-\frac{1}{2}}}%
h^{s+1}\Vert |p_{\ast }|\Vert _{2,s+1} \\
+C{\nu }^{-\frac{1}{2}}h^{k}\big(\Vert |u|\Vert _{4,k+1}^{2}+{\nu }%
^{-1}\Vert |f|\Vert _{2,\ast }+{\nu }^{-\frac{1}{2}}\Vert u_{1}^{h}\Vert +{%
\nu }^{-\frac{1}{2}}\Vert u_{0}^{h}\Vert \big) \\
+C\max_{1\leq n\leq M-1}\{(k_{n}+k_{n-1})^{2}\}\Big(\Vert u_{ttt}\Vert
_{2,0}+{\nu }^{-\frac{1}{2}}\Vert p_{tt}\Vert _{2,0}+\Vert f_{tt}\Vert _{2,0}
\\
\quad +{\nu }^{\frac{1}{2}}\Vert \nabla {u_{tt}}\Vert _{2,0}+{\nu }^{-\frac{1%
}{2}}\Vert \nabla {u_{tt}}\Vert _{4,0}^{2}+{\nu }^{-\frac{1}{2}}\Vert
|\nabla {u}|\Vert _{4,0}^{2}+{\nu }^{-\frac{1}{2}}\Vert |\nabla {u_{\ast }}%
|\Vert _{4,0}^{2}\Big).
\end{gather*}%
\end{theorem}
\noindent \textbf{Remark}: The timestep restriction \eqref{eq:timestep-restriction} comes
from discrete Gronwall inequality as it applies to the nonlinearly implicit method.
If a linearly implicit realization for the same method is used, the analysis
can be sharpened to remove the restriction \eqref{eq:timestep-restriction},
as discussed in Ingram \cite{MR3055600}.

\begin{proof}
For $\theta =1$, one-leg DLN method becomes one-leg trapezoid rule and the
conclusions of the theorem have been proved in many places, e.g., Girault and
Raviart \cite{MR548867}.
Now we consider the case $\theta \in \lbrack 0,1)$. Start with NSE at time
$t_{n,\ast }\
\left( 1\leq n\leq M-1\right) $. For any $v^{h}\in V^{h}$, the variational
formulation becomes
\begin{equation*}
(u_{t}\left( t_{n,\ast }\right) ,v^{h})+\nu (\nabla u(t_{n,\ast }),\nabla
v^{h})+b^{\ast }(u(t_{n,\ast }),u(t_{n,\ast }),v^{h})-(p(t_{n,\ast }),\nabla
\cdot v^{h})=(f(t_{n,\ast }),v^{h}).
\end{equation*}%
Equivalently
\begin{gather}
\left( \frac{{\alpha _{2}}{u_{n+1}}+{\alpha _{1}}{u_{n}}+{%
\alpha _{0}}{u_{n-1}}}{\widehat{k}_{n}},v^{h}\right)  +b^{\ast }\left(
u_{n,\ast },u_{n,\ast },v^{h}\right) +\nu \left( \nabla {u_{n,\ast }},\nabla
{v^{h}}\right)   \notag -\left( p_{n,\ast },\nabla \cdot v^{h}\right) \\
= \left( f_{n,\ast
}, v^{h}\right)  +\tau \left( u_{n,\ast },p_{n,\ast },v^{h}\right) ,
\label{eq:NSEVariation}
\end{gather}%
where the truncation error is
\begin{gather*}
\tau \left( u_{n,\ast },p_{n,\ast },v^{h}\right) =\left( \frac{{\alpha _{2}}{%
u_{n+1}}+{\alpha _{1}}{u_{n}}+{\alpha _{0}}{u_{n-1}}}{\widehat{k}_{n}}%
-u_{t}\left( t_{n,\ast }\right) ,v^{h}\right)  \\
+{\nu }\left( \nabla {\left( u_{n,\ast }-u\left( t_{n,\ast }\right) \right) }%
,\nabla {v^{h}}\right) +b^{\ast }\left( u_{n,\ast },u_{n,\ast },v^{h}\right)
-b^{\ast }\left( u\left( t_{n,\ast }\right) ,u\left( t_{n,\ast }\right)
,v^{h}\right)  \\
-\left( p_{n,\ast }-p\left( t_{n,\ast }\right) ,\nabla \cdot v^{h}\right)
+\left( f\left( t_{n,\ast }\right) -f_{n,\ast }\ ,v^{h}\right) .
\end{gather*}%
Define the finite element error $e_{n}:=u_{n}-u_{n}^{h}$ and subtract
\eqref{eq:NSEVariation} from the
one-leg DLN FEM equation \eqref{eq:DLN-NSEinfsup}
\begin{gather}
\left( \frac{{\alpha _{2}}{e_{n+1}}+{\alpha _{1}}{e_{n}}+{\alpha _{0}}{%
e_{n-1}}}{\widehat{k}_{n}},v^{h}\right) + b^{\ast }\left( u_{n,\ast
},u_{n,\ast },v^{h}\right) -b^{\ast }\left( u_{n,\ast }^{h},u_{n,\ast
}^{h},v^{h}\right) +\nu \left( \nabla {e_{n,\ast }},\nabla {v^{h}}\right)
 \notag \\
=\left( p_{n,\ast },\nabla \cdot v^{h}\right)  +  \tau \left( u_{n,\ast
},p_{n,\ast },v^{h}\right) \ \ \ \forall v^{h}\in V^{h}. \label{eq:DLN-NSEVariation}
\end{gather}%
Denote $U_{n}$ to be $L^{2}$ projection of $u_{n}$ onto $V^{h}$ and decompose $e_{n}$
as
\begin{equation*}
e_{n}=u_{n}-U_{n}-\left( u_{n}^{h}-U_{n}\right) :=\eta _{n}-\phi _{n}^{h}.
\end{equation*}%
Setting $v^{h}=\phi _{n,\ast }^{h}$, \eqref{eq:DLN-NSEVariation} writes
\begin{align*}
& \Big(\frac{{\alpha _{2}}{\phi _{n+1}^{h}}+{\alpha _{1}}{\phi _{n}^{h}}+{%
\alpha _{0}}{\phi _{n-1}^{h}}}{\widehat{k}_{n}},\phi _{n,\ast }^{h}\Big)+{%
\nu }\Vert \nabla {\phi _{n,\ast }}\Vert ^{2}+b^{\ast }(u_{n,\ast
}^{h},u_{n,\ast }^{h},\phi _{n,\ast }^{h})-b^{\ast }(u_{n,\ast },u_{n,\ast
},\phi _{n,\ast }^{h}) \\
=&\Big(\frac{{\alpha _{2}}{\eta _{n+1}^{h}}+{\alpha _{1}}{\eta _{n}^{h}}+{%
\alpha _{0}}{\eta _{n-1}^{h}}}{\widehat{k}_{n}},\phi _{n,\ast }^{h}\Big)+\nu
(\nabla {\eta _{n,\ast }},\nabla {\phi _{n,\ast }^{h}})-(p_{n,\ast },\nabla
\cdot \phi _{n,\ast }^{h})-\tau (u_{n,\ast },p_{n,\ast },\phi _{n,\ast
}^{h}).
\end{align*}%
Using $(q^{h},\nabla \cdot \phi _{n,\ast }^{h})=0$ for any $q^{h}\in Q^{h}$
and multiplying the above equation by $\widehat{k}_{n}$, we obtain
\begin{align}
& \Big(\sum_{\ell=0}^{2}{\alpha _{\ell}}{\phi _{n-1+\ell}^{h}}\ ,\phi _{n,\ast }^{h}%
\Big)+\nu \widehat{k}_{n}\Vert \nabla {\phi _{n,\ast }}\Vert ^{2}
\label{eq:erroranalysis2} \\
=&\Big(\sum_{\ell=0}^{2}{\alpha _{\ell}}{\eta _{n-1+\ell}^{h}}\ ,\phi _{n,\ast }^{h}%
\Big)+{\widehat{k}_{n}}b^{\ast }(u_{n,\ast },u_{n,\ast },\phi _{n,\ast
}^{h})-{\widehat{k}_{n}}b^{\ast }(u_{n,\ast }^{h},u_{n,\ast }^{h},\phi
_{n,\ast }^{h}) \notag \\
&+\nu {\widehat{k}_{n}}(\nabla {\eta _{n,\ast }},\nabla {\phi
_{n,\ast }^{h}})
-{\widehat{k}_{n}}(p_{n,\ast }-q^{h},\nabla \cdot \phi _{n,\ast
}^{h})-\widehat{k}_{n}\tau (u_{n,\ast },p_{n,\ast },\phi _{n,\ast
}^{h})\quad \forall q^{h}\in Q^{h}.  \notag
\end{align}

\noindent Then we analyze the terms on the right-hand side of %
\eqref{eq:erroranalysis2}. By the property of projection operators and the linearity
of inner products, we have
\begin{equation*}
({\alpha _{2}}{\eta _{n+1}^{h}}+{\alpha _{1}}{\eta _{n}^{h}}+{\alpha _{0}}{%
\eta _{n-1}^{h}},\phi _{n,\ast }^{h})=0.
\end{equation*}%
Next we apply Lemma \ref{lemma:skew-symmetry}. This yields
\begin{align*}
&{\widehat{k}_{n}}b^{\ast }(u_{n,\ast },u_{n,\ast },\phi _{n,\ast }^{h})-{%
\widehat{k}_{n}}b^{\ast }(u_{n,\ast }^{h},u_{n,\ast }^{h},\phi _{n,\ast
}^{h}) \\
=&{\widehat{k}_{n}}b^{\ast }(u_{n,\ast }-u_{n,\ast }^{h},u_{n,\ast },\phi
_{n,\ast }^{h})+{\widehat{k}_{n}}b^{\ast }(u_{n,\ast }^{h},u_{n,\ast
}-u_{n,\ast }^{h},\phi _{n,\ast }^{h}) \\
=&{\widehat{k}_{n}}b^{\ast }(\eta _{n,\ast },u_{n,\ast },\phi _{n,\ast }^{h})-%
{\widehat{k}_{n}}b^{\ast }(\phi _{n,\ast }^{h},u_{n,\ast },\phi _{n,\ast
}^{h})+{\widehat{k}_{n}}b^{\ast }(u_{n,\ast }^{h},\eta _{n,\ast },\phi
_{n,\ast }^{h}).
\end{align*}%
For any $\varepsilon >0$, using \eqref{eq:conti-triform} and Young's inequality gives
\begin{align*}
{\widehat{k}_{n}}b^{\ast }(\eta _{n,\ast },u_{n,\ast },\phi _{n,\ast
}^{h}) &\leq C(\Omega ){\widehat{k}_{n}}\Vert \eta _{n,\ast }\Vert ^{\frac{1}{2%
}}\Vert \nabla {\eta _{n,\ast }}\Vert ^{\frac{1}{2}}\Vert \nabla {u_{n,\ast }%
}\Vert \Vert \nabla {\phi _{n,\ast }^{h}}\Vert  \\
& \leq \varepsilon \nu \widehat{k}_{n}\Vert \nabla {\phi _{n,\ast }^{h}}%
\Vert ^{2}+C(\varepsilon ,\Omega )\widehat{k}_{n}\nu ^{-1}\Vert \eta
_{n,\ast }\Vert \Vert \nabla \eta _{n,\ast }\Vert \Vert \nabla u_{n,\ast
}\Vert ^{2}, \\
\widehat{k}_{n}b^{\ast }(\phi _{n,\ast }^{h},u_{n,\ast },\phi _{n,\ast
}^{h}) & \leq C(\Omega ){\widehat{k}_{n}}\Vert \phi _{n,\ast }^{h}\Vert ^{\frac{%
1}{2}}\Vert \nabla {\phi _{n,\ast }^{h}}\Vert ^{\frac{1}{2}}\Vert \nabla {%
u_{n,\ast }}\Vert \Vert \nabla {\phi _{n,\ast }^{h}}\Vert  \\
&\leq \varepsilon \nu \widehat{k}_{n}\Vert \nabla \phi _{n,\ast
}^{h}\Vert ^{2}+C(\varepsilon ,\Omega )\widehat{k}_{n}\nu ^{-3}\Vert \phi
_{n,\ast }^{h}\Vert ^{2}\Vert \nabla u_{n,\ast }\Vert ^{4}, \\
{\widehat{k}_{n}}b^{\ast }(u_{n,\ast }^{h},\eta _{n,\ast },\phi _{n,\ast
}^{h}) & \leq C(\Omega ){\widehat{k}_{n}}\Vert u_{n,\ast }^{h}\Vert ^{\frac{1}{2%
}}\Vert \nabla {u_{n,\ast }^{h}}\Vert ^{\frac{1}{2}}\Vert \nabla {\eta
_{n,\ast }}\Vert \Vert \nabla {\phi _{n,\ast }^{h}}\Vert  \\
& \leq \varepsilon \nu \widehat{k}_{n}\Vert \nabla {\phi _{n,\ast }^{h}}%
\Vert ^{2}+C(\varepsilon ,\Omega )\widehat{k}_{n}\nu ^{-1}\Vert u_{n,\ast
}^{h}\Vert \Vert \nabla {u_{n,\ast }^{h}}\Vert \Vert \nabla {\eta _{n,\ast }}%
\Vert ^{2}.
\end{align*}%
Now using the Cauchy-Schwarz and Young inequalities gives
\begin{align*}
\nu \widehat{k}_{n}(\nabla {\eta _{n,\ast }},\nabla {\phi _{n,\ast }^{h}}%
)& \leq \nu \widehat{k}_{n}\Vert \nabla {\eta _{n,\ast }}\Vert \Vert \nabla {%
\phi _{n,\ast }^{h}}\Vert \leq \varepsilon \nu \widehat{k}_{n}\Vert \nabla {%
\phi _{n,\ast }^{h}}\Vert ^{2}+C(\varepsilon )\nu \widehat{k}_{n}\Vert
\nabla {\eta _{n,\ast }}\Vert ^{2}, \\
\widehat{k}_{n}(p_{n,\ast }-q^{h},\nabla \cdot \phi _{n,\ast }^{h}) &\leq
\widehat{k}_{n}\Vert p_{n,\ast }-q^{h}\Vert \Vert \nabla \cdot \phi _{n,\ast
}^{h}\Vert
\leq \sqrt{d}{\widehat{k}_{n}}\Vert p_{n,\ast }-q^{h}\Vert \Vert \nabla {%
\phi _{n,\ast }^{h}}\Vert \\
& \leq \varepsilon \nu \widehat{k}_{n}\Vert \nabla {%
\phi _{n,\ast }^{h}}\Vert ^{2}+C(\varepsilon )\widehat{k}_{n}\nu ^{-1}\Vert
p_{n,\ast }-q^{h}\Vert ^{2},
\end{align*}%
where $d$ is the dimension of the domain $\Omega $. Now set ${\varepsilon }%
=1/10$, combine the analysis above and apply the $G$-stability relation
\eqref{eq:Gstab} to \eqref{eq:erroranalysis2}. This becomes
\begin{align*}
& \begin{Vmatrix}
\phi _{n+1}^{h} \\
\phi _{n}^{h}%
\end{Vmatrix}%
_{G(\theta )}^{2}-%
\begin{Vmatrix}
\phi _{n}^{h} \\
\phi _{n-1}^{h}%
\end{Vmatrix}%
_{G(\theta )}^{2}+\frac{\nu }{2}{\widehat{k}_{n}}\left\Vert \nabla {\phi
_{n,\ast }}\right\Vert ^{2}+\left\Vert \sum_{\ell=0}^{2}a_{\ell}^{(n)}{\phi
_{n-1+\ell}^{h}}\right\Vert ^{2} \\
\leq & C{\frac{\widehat{k}_{n}}{{\nu }^{3}}}\Vert \phi _{n,\ast }^{h}\Vert
^{2}\Vert \nabla {u_{n,\ast }}\Vert ^{4}+C{\nu }{\widehat{k}_{n}}\Vert
\nabla {\eta _{n,\ast }}\Vert ^{2}+C{\frac{{\widehat{k}_{n}}}{\nu }}\Vert
\eta _{n,\ast }\Vert \Vert \nabla {\eta _{n,\ast }}\Vert \Vert \nabla {%
u_{n,\ast }}\Vert ^{2} \\
&+C{\frac{\widehat{k}_{n}}{\nu }}\Vert u_{n,\ast }^{h}\Vert \Vert
\nabla {u_{n,\ast }^{h}}\Vert \Vert \nabla {\eta _{n,\ast }}\Vert ^{2}+C{%
\frac{\widehat{k}_{n}}{\nu }}\Vert p_{n,\ast }-q^{h}\Vert ^{2}+{\widehat{k}%
_{n}}\left\vert \tau \left( u_{n,\ast },p_{n,\ast },\phi _{n,\ast
}^{h}\right) \right\vert .
\end{align*}%
Summing up from $n=1$ to $n=M-1$, we have
\begin{gather}
\begin{Vmatrix}
\phi _{M}^{h} \\
\phi _{M-1}^{h}%
\end{Vmatrix}%
_{G(\theta )}^{2}-%
\begin{Vmatrix}
\phi _{1}^{h} \\
\phi _{0}^{h}%
\end{Vmatrix}%
_{G(\theta )}^{2}+\sum_{n=1}^{M-1}\left\Vert
\sum_{\ell=0}^{2}a_{\ell}^{(n)}{\phi _{n-1+\ell}^{h}}\right\Vert ^{2}+\frac{\nu }{2}%
\sum_{n=1}^{M-1}{\widehat{k}_{n}}\Vert \nabla {\phi _{n,\ast }}\Vert ^{2}
\label{eq:DLNErr-noGnorm} \\
\leq  \sum_{n=1}^{M-1}C{\frac{\widehat{k}_{n}}{{\nu }^{3}}}\Vert \phi
_{n,\ast }^{h}\Vert ^{2}\Vert \nabla {u_{n,\ast }}\Vert
^{4}+\sum_{n=1}^{M-1}C{\nu }{\widehat{k}_{n}}\Vert \nabla {\eta _{n,\ast }}%
\Vert ^{2}+\sum_{n=1}^{M-1}C{\frac{{\widehat{k}_{n}}}{\nu }}\Vert \eta
_{n,\ast }\Vert \Vert \nabla {\eta _{n,\ast }}\Vert \Vert \nabla {u_{n,\ast }%
}\Vert ^{2} \notag \\
+\sum_{n=1}^{M-1}C{\frac{\widehat{k}_{n}}{\nu }}\Vert u_{n,\ast
}^{h}\Vert \Vert \nabla {u_{n,\ast }^{h}}\Vert \Vert \nabla {\eta _{n,\ast }}%
\Vert ^{2}+\sum_{n=1}^{M-1}C{\frac{\widehat{k}_{n}}{\nu }}\Vert p_{n,\ast
}-q^{h}\Vert ^{2}+\sum_{n=1}^{M-1}{\widehat{k}_{n}}\left\vert \tau \left(
u_{n,\ast },p_{n,\ast },\phi _{n,\ast }^{h}\right) \right\vert .  \notag
\end{gather}%
\noindent Set the approximate solution of $u$ at two initial time-steps $t_{0}$
and $t_{1}$ to be $L^{2}$ projection of $u$ into $V^{h}$. We have
\begin{equation*}
\phi _{i}^{h}=u_{i}^{h}-U_{i}=0,\ i=0,1.
\end{equation*}%
\noindent Using the definition of the $G$-norm \eqref{eq:G-norm},
the estimate \eqref{eq:DLNErr-noGnorm} becomes
\begin{gather}
 {\frac{1}{4}}(1+\theta )\Vert \phi _{M}^{h}\Vert ^{2}+{\frac{1}{4}}%
(1-\theta )\Vert \phi _{M-1}^{h}\Vert ^{2}+\sum_{n=1}^{M-1}%
\left\Vert \sum_{\ell=0}^{2}a_{\ell}^{(n)}{\phi _{n-1+\ell}^{h}}\right\Vert ^{2}+%
\frac{\nu }{2}\sum_{n=1}^{M-1}{\widehat{k}_{n}}\Vert \nabla {\phi _{n,\ast }}%
\Vert ^{2} \label{eq:Errpreliminary} \\
\leq  \sum_{n=1}^{M-1}C{\frac{\widehat{k}_{n}}{{\nu }^{3}}}\Vert \phi
_{n,\ast }^{h}\Vert ^{2}\Vert \nabla {u_{n,\ast }}\Vert
^{4}+\sum_{n=1}^{M-1}C{\nu }{\widehat{k}_{n}}\Vert \nabla {\eta _{n,\ast }}%
\Vert ^{2}+\sum_{n=1}^{M-1}C{\frac{{\widehat{k}_{n}}}{\nu }}\Vert \eta
_{n,\ast }\Vert \Vert \nabla {\eta _{n,\ast }}\Vert \Vert \nabla {u_{n,\ast }%
}\Vert ^{2}  \notag \\
 +\sum_{n=1}^{M-1}C{\frac{\widehat{k}_{n}}{\nu }}\Vert u_{n,\ast
}^{h}\Vert \Vert \nabla {u_{n,\ast }^{h}}\Vert \Vert \nabla {\eta _{n,\ast }}%
\Vert ^{2}+\sum_{n=1}^{M-1}C{\frac{\widehat{k}_{n}}{\nu }}\Vert p_{n,\ast
}-q^{h}\Vert ^{2}+\sum_{n=1}^{M-1}{\widehat{k}_{n}}\left\vert \tau \left(
u_{n,\ast },p_{n,\ast },\phi _{n,\ast }^{h}\right) \right\vert .  \notag
\end{gather}%
By the uniform continuity of functions ${\beta _{l}^{(n)}}\left( \varepsilon
_{n},\theta \right) $ ($l=0,1,2$), we have
\begin{gather}  \label{eq:Est-gradeta}
\left\Vert \nabla {\eta _{n,\ast }}\right\Vert =\left\Vert \left( \nabla {%
\sum_{\ell=0}^{2}{\beta _{\ell}^{(n)}}{\eta _{n-1+\ell}}}\right) \right\Vert \leq
\sum_{\ell=0}^{2}\left\vert {\beta _{\ell}^{(n)}}\right\vert \left\Vert \nabla {%
\eta _{n-1+\ell}}\right\Vert \leq C\sum_{\ell=0}^{2}\left\Vert \nabla {\eta
_{n-1+\ell}}\right\Vert .
\end{gather}%
Using the interpolation error estimates \eqref{eq:interpolation-Error},
\eqref{eq:Est-gradeta} yields
\begin{align*}
\sum_{n=1}^{M-1}C{\nu }{\widehat{k}_{n}}\Vert \nabla {\eta _{n,\ast }}%
\Vert ^{2} &\leq C{\nu }\sum_{n=1}^{M-1}{\widehat{k}_{n}}\sum_{\ell=0}^{2}\Vert
\nabla {\eta _{n-1+\ell}}\Vert ^{2}\leq C{\nu }{h^{2k}}\sum_{n=1}^{M-1}{%
\widehat{k}_{n}}\sum_{\ell=0}^{2}\Vert u_{n-1+\ell}\Vert _{k+1}^{2} \notag \\
& \leq C\left( \theta \right) {\nu }{h^{2k}}\sum_{n=1}^{M-1}\left(
k_{n}+k_{n-1}\right) \sum_{\ell=0}^{2}\Vert u_{n-1+\ell}\Vert _{k+1}^{2}, \notag
\end{align*}%
for some constant $C\left( \theta \right) $. Using now Lemma \ref{lemma:disnormEst},
this implies
\begin{gather} \label{eq:EstimateEta}
\sum_{n=1}^{M-1}C{\nu }{\widehat{k}_{n}}\Vert \nabla {\eta _{n,\ast }}\Vert
^{2}\leq C\left( \theta \right) {\nu }{h^{2k}}{\left\Vert \left\vert
u\right\vert \right\Vert _{2,k+1}}^{2}.
\end{gather}%
Using again the uniform continuity of $\{ {\beta _{\ell}^{(n)}}\} _{\ell=0}^{2}$ and
the estimates \eqref{eq:interpolation-Error}, we have
\begin{align*}
\left\Vert \eta _{n,\ast }\right\Vert \left\Vert \nabla {\eta _{n,\ast }}%
\right\Vert & =\left\Vert \sum_{\ell=0}^{2}{\beta _{\ell}^{(n)}}{\eta _{n-1+\ell}}\right\Vert
\left\Vert \nabla {\left( \sum_{\ell=0}^{2}{\beta _{\ell}^{(n)}}{\eta _{n-1+\ell}}\right) }%
\right\Vert  \\
& \leq C\left( \sum_{0\leq i,j\leq 2}\Vert \eta _{n-1+i}\Vert \Vert \nabla {%
\eta _{n-1+j}}\Vert \right) \\
& \leq Ch^{2k+1}\sum_{0\leq i,j\leq 2}\Vert
u_{n-1+i}\Vert _{k+1}\Vert \nabla {u_{n-1+j}}\Vert _{k+1}.
\end{align*}%
Similarly,
\begin{equation*}
\left\Vert \nabla {u_{n,\ast }}\right\Vert ^{2}=\left\Vert \nabla \left(
\sum_{\ell=0}^{2}{\beta _{\ell}^{(n)}}{u_{n-1+\ell}}\right) \right\Vert ^{2}\leq
C\sum_{\ell=0}^{2}\left\Vert \nabla {u_{n-1+\ell}}\right\Vert ^{2}.
\end{equation*}%
Thus by Young's inequality and Lemma \ref{lemma:disnormEst}, we have
\begin{align*}
& \sum_{n=1}^{M-1}C{\frac{{\widehat{k}_{n}}}{\nu }}\Vert \eta _{n,\ast
}\Vert \Vert \nabla {\eta _{n,\ast }}\Vert \Vert \nabla {u_{n,\ast }}\Vert
^{2} \notag \\
\leq & C{\nu }^{-1}h^{2k+1}\sum_{n=1}^{M-1}\widehat{k}_{n}\left( \sum_{0\leq
i,j\leq 2}\Vert u_{n-1+i}\Vert _{k+1}\Vert \nabla {u_{n-1+j}}\Vert
_{k+1}\right) \left( \sum_{\ell=0}^{2}\left\Vert \nabla {u_{n-1+\ell}}\right\Vert
^{2}\right) \notag \\
\leq & C\left( \theta \right) {\nu }^{-1}h^{2k+1}\sum_{n=1}^{M-1}\left(
k_{n}+k_{n-1}\right) \left( \sum_{\ell=0}^{2}\left\Vert u_{n-1+\ell}\right\Vert
_{k+1}^{4}+\sum_{\ell=0}^{2}\left\Vert \nabla u_{n-1+\ell}\right\Vert ^{4}\right)
\notag \\
\leq & C\left( \theta \right) {\nu }^{-1}h^{2k+1}\left( \left\Vert
\left\vert u\right\vert \right\Vert _{4,k+1}^{4}+\left\Vert \left\vert
\nabla {u}\right\vert \right\Vert _{4,0}^{4}\right) . \notag
\end{align*}%
Recall that by Theorem \ref{theorem:DLN-NSEStab}, we have an priori bound for $\left\Vert
u_{n}^{h}\right\Vert $($n=2,3,\cdots M$). Then combine \eqref{eq:interpolation-Error}
and Young's Inequality. This yields
\begin{align*}
& \sum_{n=1}^{M-1}C{\frac{\widehat{k}_{n}}{\nu }}\Vert u_{n,\ast }^{h}\Vert
\Vert \nabla {u_{n,\ast }^{h}}\Vert \Vert \nabla {\eta _{n,\ast }}\Vert
^{2}\leq C{\nu }^{-1}\sum_{n=1}^{M-1}\widehat{k}_{n}\left\Vert \nabla {%
u_{n,\ast }^{h}}\right\Vert \left\Vert \nabla {\eta _{n,\ast }}\right\Vert
^{2} \notag \\
\leq &C\left( \theta \right) {\nu }^{-1}h^{2k}\left( \sum_{n=1}^{M-1}\widehat{%
k}_{n}\left( \sum_{\ell=0}^{2}\left\Vert u_{n-1+\ell}\right\Vert _{k+1}^{4}\right)
+\sum_{n=1}^{M-1}\widehat{k}_{n}\left\Vert \nabla {u_{n,\ast }^{h}}%
\right\Vert ^{2}\right) .  \notag
\end{align*}%
By the $LBB^{h}$ condition, $f\left( t_{n,\ast }\right) $ can be replace by $%
f_{n,\ast }$ in Theorem \ref{theorem:DLN-NSEStab}. Now we apply
Theorem \ref{theorem:DLN-NSEStab} to bound $\widehat{k}_{n}\left\Vert \nabla {u_{n,\ast }^{h}}
\right\Vert ^{2}$, which yields
\begin{equation*}
\sum_{n=1}^{M-1}\widehat{k}_{n}\left\Vert \nabla {u_{n,\ast }^{h}}%
\right\Vert ^{2}\leq \sum_{n=1}^{M-1}{\frac{1}{{\nu }^{2}}}{\widehat{k}_{n}}%
\left\Vert f_{n,\ast }\right\Vert _{\ast }^{2}+\frac{1}{\nu }\left\Vert
u_{1}^{h}\right\Vert ^{2}+\frac{1}{\nu }\left\Vert u_{0}^{h}\right\Vert ^{2}.
\end{equation*}%
Applying Lemma \ref{lemma:disnormEst} again, the above two inequalities imply
\begin{align}
&\sum_{n=1}^{M-1}C\left( \Omega \right) {\frac{\widehat{k}_{n}}{\nu }}\Vert
u_{n,\ast }^{h}\Vert \Vert \nabla {u_{n,\ast }^{h}}\Vert \Vert \nabla {\eta
_{n,\ast }}\Vert ^{2}   \notag \\
\leq &C\left( \theta \right) {\nu }^{-1}h^{2k}\left( \left\Vert \left\vert
u\right\vert \right\Vert _{4,k+1}^{4}+ \frac{1}{\nu^{2}}\left\Vert \left\vert f\right\vert
\right\Vert _{2,\ast }^{2}+\frac{1}{\nu }\left\Vert u_{1}^{h}\right\Vert
^{2}+\frac{1}{\nu }\left\Vert u_{0}^{h}\right\Vert ^{2}\right) . \label{eq:Err-pressureinterpo}
\end{align}%
Using the interpolation error estimate for pressure $p$,
we have
\begin{align}
\sum_{n=1}^{M-1}C{\frac{\widehat{k}_{n}}{\nu }}\left\Vert p_{n,\ast
}-q^{h}\right\Vert ^{2}
& \leq C{\nu }^{-1}\left( \sum_{n=1}^{M-1}\widehat{k}_{n}\left\Vert
p_{n,\ast }-p\left( t_{n,\ast }\right) \right\Vert ^{2}+\sum_{n=1}^{M-1}%
\widehat{k}_{n}\left\Vert p\left( t_{n,\ast }\right) -q^{h}\right\Vert
^{2}\right)   \notag \\
& \leq  C{\nu }^{-1}\left( \sum_{n=1}^{M-1}\widehat{k}_{n}\left\Vert p_{n,\ast }-p\left(
t_{n,\ast }\right) \right\Vert ^{2}+h^{2s+2}\left\Vert \left\vert p_{\ast
}\right\vert \right\Vert _{2,s+1}^{2}\right) , \label{eq:Err-pressure}
\end{align}%
and using the consistency errors Lemma \ref{lemma:consistError} yields
\begin{align}
\sum_{n=1}^{M-1}\widehat{k}_{n}\left\Vert p\left( t_{n,\ast }\right)
-q^{h}\right\Vert ^{2} &\leq C\sum_{n=1}^{M-1}{\widehat{k}_{n}}\left(
k_{n}+k_{n-1}\right) ^{3}\int_{t_{n-1}}^{t_{n+1}}\Vert p_{tt}\Vert ^{2}dt
\notag \\
& \leq C\left( \theta \right) \max_{1\leq n\leq M-1}{\left\{ \left(
k_{n}+k_{n-1}\right) ^{4}\right\} }\Vert p_{tt}\Vert _{2,0}^{2}. \notag
\end{align}%
We combine \eqref{eq:Err-pressureinterpo} and \eqref{eq:Err-pressure} to obtain
\begin{gather} \label{eq:InterpolationPressure}
\sum_{n=1}^{M-1}C{\frac{\widehat{k}_{n}}{\nu }}\left\Vert p_{n,\ast
}-q^{h}\right\Vert ^{2} \\
\leq C\left( \theta \right) {\nu }^{-1}\left( h^{2s+2}{\left\Vert \left\vert
p_{\ast }\right\vert \right\Vert _{2,s+1}}^{2}+\max_{1\leq n\leq M-1}{%
\left\{ \left( k_{n}+k_{n-1}\right) ^{4}\right\} }\Vert p_{tt}\Vert
_{2,0}^{2}\right) .
\end{gather}%
\newline

Let us now treat the truncation error
$\left\vert \tau \left( u_{n,\ast },p_{n,\ast },\phi _{n,\ast
}^{h}\right) \right\vert $. Using the Cauchy-Schwarz inequality, we have
\begin{equation*}
\left( \frac{\sum_{\ell=0}^{2}{\alpha _{\ell}}{u_{n-1+\ell}}}{\widehat{k}_{n}}%
-u_{t}\left( t_{n,\ast }\right) ,\phi _{n,\ast }^{h}\right) \leq {\frac{1}{2}%
}\left\Vert \phi _{n,\ast }^{h}\right\Vert ^{2}+{\frac{1}{2}}\left\Vert
\frac{\sum_{\ell=0}^{2}{\alpha _{\ell}}{u_{n-1+\ell}}}{\widehat{k}_{n}}-u_{t}\left(
t_{n,\ast }\right) \right\Vert ^{2},
\end{equation*}%
and applying again Lemma \ref{lemma:consistError}, for $\theta \in \lbrack 0,1)$
to the last term above
\begin{equation*}
\sum_{n=1}^{M-1}\widehat{k}_{n}\left\Vert \frac{\sum_{\ell=0}^{2}{\alpha _{\ell}}{%
u_{n-1+\ell}}}{\widehat{k}_{n}}-u_{t}\left( t_{n,\ast }\right) \right\Vert
^{2}\leq C\left( \theta \right) \max_{1\leq n\leq M-1}\left\{ \left(
k_{n}+k_{n-1}\right) ^{4}\right\} \Vert u_{ttt}\Vert _{2,0}^{2},
\end{equation*}%
we have
\begin{gather}
\sum_{n=1}^{M-1}\widehat{k}_{n}\left( \frac{\sum_{\ell=0}^{2}{\alpha _{\ell}}{%
u_{n-1+\ell}}}{\widehat{k}_{n}}-u_{t}\left( t_{n,\ast }\right) ,\phi _{n,\ast
}^{h}\right) \notag \\
\leq \frac{1}{2}\sum_{n=1}^{M-1}\widehat{k}_{n}\left\Vert \phi _{n,\ast
}^{h}\right\Vert ^{2}+C\left( \theta \right) \max_{1\leq n\leq M-1}\left\{
\left( k_{n}+k_{n-1}\right) ^{4}\right\} \left\Vert u_{ttt}\right\Vert
_{2,0}^{2}. \notag
\end{gather}%
Similarly,
\begin{equation*}
\sum_{n=1}^{M-1}\widehat{k}_{n}\left( f\left( t_{n,\ast }\right) -f_{n,\ast
}\ ,\phi _{n,\ast }^{h}\right) \leq \frac{1}{2}\sum_{n=1}^{M-1}\widehat{k}%
_{n}\Vert \phi _{n,\ast }^{h}\Vert ^{2}+C\left( \theta \right) \max_{1\leq
n\leq M-1}\left\{ \left( k_{n}+k_{n-1}\right) ^{4}\right\} \Vert f_{tt}\Vert
_{2,0}^{2},
\end{equation*}%
and also
\begin{align*}
&\sum_{n=1}^{M-1}\widehat{k}_{n}\left( p_{n,\ast }-p\left( t_{n,\ast }\right)
,\nabla \cdot \phi _{n,\ast }^{h}\right)  \\
\leq & \varepsilon \nu \sum_{n=1}^{M-1}\widehat{k}_{n}\Vert \nabla {\phi
_{n,\ast }^{h}}\Vert ^{2}+C\left( \varepsilon ,\theta \right) {\nu }%
^{-1}\max_{1\leq n\leq M-1}\left( \left( k_{n}+k_{n-1}\right) ^{4}\right)
\left\Vert p_{tt}\right\Vert _{2,0}^{2}, \\
& \sum_{n=1}^{M-1}{\nu }\widehat{k}_{n}\left( \nabla {\left( u_{n,\ast
}-u\left( t_{n,\ast }\right) \right) },\nabla {\phi _{n,\ast }^{h}}\right)
\\
\leq & \varepsilon \nu \sum_{n=1}^{M-1}\widehat{k}_{n}\Vert \nabla {\phi
_{n,\ast }^{h}}\Vert ^{2}+C\left( \varepsilon ,\theta \right) {\nu }%
\max_{1\leq n\leq M-1}\left\{ \left( k_{n}+k_{n-1}\right) ^{4}\right\}
\left\Vert \nabla {u_{tt}}\right\Vert _{2,0}^{2}.
\end{align*}%
Moreover
\begin{align*}
& b^{\ast }\left( u_{n,\ast },u_{n,\ast },\phi _{n,\ast }^{h}\right)
-b^{\ast }\left( u\left( t_{n,\ast }\right) ,u\left( t_{n,\ast }\right)
,\phi _{n,\ast }^{h}\right)  \\
= &b^{\ast }\left( u_{n,\ast }-u\left( t_{n,\ast }\right) ,u_{n,\ast
},\phi _{n,\ast }^{h}\right) +b^{\ast }\left( u\left( t_{n,\ast }\right)
,u_{n,\ast }-u\left( t_{n,\ast }\right) ,\phi _{n,\ast }^{h}\right)  \\
\leq & C\Vert \nabla \left( u_{n,\ast }-u\left( t_{n,\ast }\right)
\right) \Vert \Vert \nabla {\phi _{n,\ast }^{h}}\Vert \left( \Vert \nabla {%
u_{n,\ast }}\Vert +\Vert \nabla {u\left( t_{n,\ast }\right) }\Vert \right)
\\
\leq & \varepsilon \nu \Vert \nabla {\phi _{n,\ast }^{h}}\Vert
^{2}+C\left( \varepsilon \right) {\nu }^{-1}\Vert \nabla \left( u_{n,\ast
}-u\left( t_{n,\ast }\right) \right) \Vert ^{2}\left( \Vert \nabla {%
u_{n,\ast }}\Vert ^{2}+\Vert \nabla {u\left( t_{n,\ast }\right) }\Vert
^{2}\right) , \\
& \left\Vert \nabla \left( u_{n,\ast }-u\left( t_{n,\ast }\right) \right)
\right\Vert ^{2}\left( \left\Vert \nabla {u_{n,\ast }}\right\Vert
^{2}+\left\Vert \nabla {u\left( t_{n,\ast }\right) }\right\Vert ^{2}\right)
\\
\leq & C\left( \left\Vert \nabla {u_{n,\ast }}\right\Vert
^{2}+\left\Vert \nabla {u\left( t_{n,\ast }\right) }\right\Vert ^{2}\right)
\left( k_{n}+k_{n-1}\right) ^{3}\int_{t_{n-1}}^{t_{n+1}}\left\Vert \nabla {%
u_{tt}}\right\Vert ^{2}dt \\
\leq & C\left( k_{n}+k_{n-1}\right)
^{3}\int_{t_{n-1}}^{t_{n+1}}\left( \left\Vert \nabla {u_{n,\ast }}%
\right\Vert ^{2}+\left\Vert \nabla {u\left( t_{n,\ast }\right) }\right\Vert
^{2}\right) \left\Vert \nabla {u_{tt}}\right\Vert ^{2}dt \\
\leq & C\left( k_{n}+k_{n-1}\right)
^{3}\int_{t_{n-1}}^{t_{n+1}}\left( \left\Vert \nabla {u_{n,\ast }}%
\right\Vert ^{4}+\left\Vert \nabla {u\left( t_{n,\ast }\right) }\right\Vert
^{4}+\left\Vert \nabla {u_{tt}}\right\Vert ^{4}\right) dt \\
\leq & C\left( k_{n}+k_{n-1}\right) ^{4}\left( \Vert \nabla {%
u_{n,\ast }}\Vert ^{4}+\Vert \nabla {u\left( t_{n,\ast }\right) }\Vert
^{4}\right) +C\left( k_{n}+k_{n-1}\right) ^{3}\int_{t_{n-1}}^{t_{n+1}}\Vert
\nabla {u_{tt}}\Vert ^{4}dt.
\end{align*}%
\noindent Now combine Lemma \ref{lemma:disnormEst} and Lemma \ref{lemma:consistError}.
This yields
\begin{align*}
&\sum_{n=1}^{M-1}\widehat{k}_{n}\left( b^{\ast }\left( u_{n,\ast
},u_{n,\ast },\phi _{n,\ast }^{h}\right) -b^{\ast }\left( u\left( t_{n,\ast
}\right) ,u\left( t_{n,\ast }\right) ,\phi _{n,\ast }^{h}\right) \right)  \\
\leq & \varepsilon \nu \sum_{n=1}^{M-1}\widehat{k}_{n}\left\Vert \nabla {%
\phi _{n,\ast }^{h}}\right\Vert ^{2}+\frac{C\left( \varepsilon ,\theta
\right) }{\nu }\max_{1\leq n\leq M-1}\left( k_{n}+k_{n-1}\right)
^{4}\left\Vert \nabla {u_{tt}}\right\Vert _{4,0}^{4} \\
& +\frac{C\left( \varepsilon ,\theta \right) }{\nu }\max_{1\leq n\leq
M-1}\left\{ \left( k_{n}+k_{n-1}\right) ^{4}\right\} \left( \sum_{n=1}^{M-1}%
\widehat{k}_{n}\left\Vert \nabla {u_{n,\ast }}\right\Vert
^{4}+\sum_{n=1}^{M-1}\widehat{k}_{n}\left\Vert \nabla {u\left( t_{n,\ast
}\right) }\right\Vert ^{4}\right)  \\
\leq & \varepsilon \nu \sum_{n=1}^{M-1}\widehat{k}_{n}\left\Vert \nabla {%
\phi _{n,\ast }^{h}}\right\Vert ^{2}+\frac{C\left( \varepsilon ,\theta
\right) }{\nu }\max_{1\leq n\leq M-1}\left( k_{n}+k_{n-1}\right)
^{4}\left\Vert \nabla {u_{tt}}\right\Vert _{4,0}^{4} \\
& +\frac{C\left( \varepsilon ,\theta \right) }{\nu }\max_{1\leq n\leq
M-1}\left\{ \left( k_{n}+k_{n-1}\right) ^{4}\right\} \left( \left\Vert
\left\vert \nabla {u}\right\vert \right\Vert _{4,0}^{4}+\left\Vert
\left\vert \nabla {u_{\ast }}\right\vert \right\Vert _{4,0}^{4}\right) .
\end{align*}%
Setting $\varepsilon =1/12$ and obtain the following estimate for the truncation
error term
\begin{align} \label{eq:EstimateTruncationError}
& \sum_{n=1}^{M-1}{\widehat{k}_{n}}\left\vert \tau \left( u_{n,\ast
},p_{n,\ast },\phi _{n,\ast }^{h}\right) \right\vert
\leq \sum_{n=1}^{M-1}{\widehat{k}_{n}}\Vert \phi _{n,\ast }^{h}\Vert ^{2}+{%
\frac{1}{4}}{\nu }\sum_{n=1}^{M-1}{\widehat{k}_{n}}\Vert \nabla {\phi
_{n,\ast }^{h}}\Vert ^{2} \\
&+C\left( \theta \right) \max_{1\leq n\leq
M-1}\left( k_{n}+k_{n-1}\right) ^{4}\bigg[\Vert u_{ttt}\Vert _{2,0}^{2}
+{\nu}^{-1}\Vert p_{tt}\Vert _{2,0}^{2}
+\Vert f_{tt}\Vert _{2,0}^{2}+\nu \Vert \nabla {u_{tt}}\Vert _{2,0}^{2} \notag \\
&+\frac{1}{\nu}\Vert \nabla {u_{tt}}\Vert _{4,0}^{4}
+\frac{1}{\nu} \Big({\Vert |\nabla {u}%
|\Vert _{4,0}}^{4}+{\Vert |\nabla {u_{\ast }}|\Vert
_{4,0}}^{4}\big)\bigg]. \notag
\end{align}%
\newline

\noindent Now we collect the terms from \eqref{eq:Errpreliminary}, \eqref{eq:EstimateEta},
\eqref{eq:Err-pressureinterpo},
\eqref{eq:InterpolationPressure}, \eqref{eq:EstimateTruncationError} and define
\begin{align*}
& \ \ \ \widetilde{F} \left( h, \max_{1\leq n\leq M-1}\left( k_{n}+k_{n-1}\right)
\right) \\
&=C(\theta) \left( {\nu }h^{2k}\left\Vert \left\vert
u\right\vert \right\Vert _{2,k+1}^{2}+\frac{h^{2k+1}}{\nu}
\left( \left\Vert \left\vert u\right\vert \right\Vert
_{4,k+1}^{4} +\left\Vert \left\vert \nabla {u}\right\vert \right\Vert
_{4,0}^{4}\right) +\frac{h^{2s+2}}{\nu}\Vert |p_{\ast }|\Vert
_{2,s+1}^{2} \right)\\
&+C(\theta )\frac{h^{2k}}{\nu}\Big(\left\Vert \left\vert
u\right\vert \right\Vert _{4,k+1}^{4}+ \frac{1}{{\nu }^{2}}\left\Vert \left\vert
f\right\vert \right\Vert _{2,\ast }^{2}+ \frac{1}{\nu} \Vert u_{1}^{h}\Vert
^{2}+\frac{1}{\nu} \Vert u_{0}^{h}\Vert ^{2}\Big) \\
&+C(\theta )\max_{1\leq n\leq M-1}(k_{n}+k_{n-1})^{4}\Big(\left\Vert
u_{ttt}\right\Vert _{2,0}^{2}+\frac{1}{\nu}\left\Vert p_{tt}\right\Vert
_{2,0}^{2}+\left\Vert f_{tt}\right\Vert _{2,0}^{2}+{\nu }\left\Vert \nabla {%
u_{tt}}\right\Vert _{2,0}^{2} \\
& +\frac{1}{\nu}\left\Vert \nabla {u_{tt}}\right\Vert _{4,0}^{4}+
\frac{1}{\nu}\left\Vert \left\vert \nabla {u}\right\vert \right\Vert _{4,0}^{4}+%
\frac{1}{\nu}\left\Vert \left\vert \nabla {u_{\ast }}\right\vert \right\Vert
_{4,0}^{4}\Big).
\end{align*}%
Thus \eqref{eq:Errpreliminary} becomes
\begin{align} \label{eq:Err-final}
{\frac{1}{4}}\Vert \phi _{M}^{h}\Vert ^{2}+\frac{\nu }{4}\sum_{n=1}^{M-1}{%
\widehat{k}_{n}}\Vert \nabla {\phi _{n,\ast }^{h}}\Vert ^{2}
\leq & \sum_{n=1}^{M-1}\left( C\left( \theta \right) {{\nu }^{-3}}\Vert
\nabla {u_{n,\ast }}\Vert ^{4}+1\right) \widehat{k}_{n}\Vert \phi
_{n,\ast }^{h}\Vert ^{2} \notag \\
& +\widetilde{F}\big(h,\max_{1\leq n\leq M-1}(k_{n}+k_{n-1})\big).
\end{align}%
For convenience, we define the sequence $\{D_{n}\}_{n=1}^{M-1}$
\begin{gather*}
D_{n}: = \left( C\left( \theta \right) {{\nu }^{-3}}\Vert
\nabla {u_{n,\ast }}\Vert ^{4}+1\right) \widehat{k}_{n}, \ \ \ n = 1, \cdots, M-1,
\end{gather*}
and the sequence $\{d_{n}\}_{n=0}^{M}$
\begin{align*}
&d_{0}: = D_{1},\  d_{1}:= D_{1} + D_{2},\  d_{M-1}:= D_{M-2} + D_{M-1},\ d_{M}:= D_{M-1},\\
&d_{n}:= \sum_{\ell=0}^{2} D_{n-1+\ell}(2 \leq n \leq M-2), \ \ \ 2 \leq n \leq M-2.
\end{align*}
We use the triangle inequality in \eqref{eq:Err-final} to obtain
\begin{align*}
\left\Vert \phi _{M}^{h}\right\Vert ^{2}+{\nu }\sum_{n=1}^{M-1}{\widehat{k}%
_{n}}\Vert \nabla {\phi _{n,\ast }}\Vert ^{2} \leq C \left( \theta \right)
\sum_{n=0}^{M}d_{n} \Vert\phi_{n}^{h}
\Vert^{2} +\widetilde{F}\big(h,\max_{1\leq n\leq M-1}(k_{n}+k_{n-1})\big),
\end{align*}
then apply the discrete Gronwall inequality (Lemma \ref{lemma:discrete-Gronwall}) under
the timestep condition \eqref{eq:timestep-restriction}
\begin{gather} \label{eq:Err-disGronwall}
\left\Vert \phi _{M}^{h}\right\Vert ^{2}+{\nu }\sum_{n=1}^{M-1}{\widehat{k}%
_{n}}\Vert \nabla {\phi _{n,\ast }}\Vert ^{2}\leq \exp \left(\sum_{n=1}^{M-1}
C \left(\theta \right) \frac{d_{n}}{1 - d_{n}}
\right)\widetilde{F}\big(h,\max_{1\leq n\leq M-1}(k_{n}+k_{n-1})\big).
\end{gather}%
Define
\begin{align*}
& F\big(h,\max_{1\leq n\leq M-1}(k_{n}+k_{n-1})\big)=C(\theta )\nu ^{\frac{1%
}{2}}h^{k}\Vert |u|\Vert _{2,k+1} \\
& \quad +C\left( \theta \right) {\nu }^{-\frac{1}{2}}h^{k+\frac{1}{2}}\left(
\left\Vert \left\vert u\right\vert \right\vert _{4,k+1}^{2}+\left\Vert
\left\vert \nabla {u}\right\vert \right\vert _{4,0}^{2}\right) +C\left(
\theta \right) {{\nu }^{-\frac{1}{2}}}h^{s+1}\left\Vert \left\vert p_{\ast
}\right\vert \right\Vert _{2,s+1} \\
& \quad +C\left( \theta \right) {\nu }^{-\frac{1}{2}}h^{k}\left( \left\Vert
\left\vert u\right\vert \right\Vert _{4,k+1}^{2}+{\nu }^{-1}\left\Vert
\left\vert f\right\vert \right\Vert _{2,\ast }+{\nu }^{-\frac{1}{2}}\Vert
u_{1}^{h}\Vert +{\nu }^{-\frac{1}{2}}\Vert u_{0}^{h}\Vert \right)  \\
& \quad +C\left( \theta \right) \max_{1\leq n\leq M-1}\{\left(
k_{n}+k_{n-1}\right) ^{2}\}\left( \left\Vert u_{ttt}\right\Vert _{2,0}+\nu
^{-\frac{1}{2}}\left\Vert p_{tt}\right\Vert _{2,0}+\left\Vert
f_{tt}\right\Vert _{2,0}\right.  \\
& \quad \left. +{\nu }^{\frac{1}{2}}\left\Vert \nabla {u_{tt}}\right\Vert
_{2,0}+{\nu }^{-\frac{1}{2}}\left\Vert \nabla {u_{tt}}\right\Vert _{4,0}^{2}+%
{\nu }^{-\frac{1}{2}}\left\Vert \left\vert \nabla {u}\right\vert \right\Vert
_{4,0}^{2}+{\nu }^{-\frac{1}{2}}\left\Vert \left\vert \nabla {u_{\ast }}%
\right\vert \right\Vert _{4,0}^{2}\right) .
\end{align*}%
Then from \eqref{eq:Err-disGronwall} we have
\begin{gather}  \label{eq:Err-finaldisGronwall}
\left\Vert \phi _{M}^{h}\right\Vert \leq F\left( h,\max_{1\leq n\leq
M-1}\left( k_{n}+k_{n-1}\right) \right) .
\end{gather}%
Combining \eqref{eq:interpolation-Error} and \eqref{eq:Err-finaldisGronwall}
yields
\begin{align*}
\left\Vert \left\vert u-u^{h}\right\vert \right\Vert _{\infty
,0}&:=\max_{0\leq n\leq M}\left\Vert u_{n}-u_{n}^{h}\right\Vert \leq
\max_{0\leq n\leq M}\left\Vert \eta _{n}\right\Vert +\max_{0\leq n\leq
M}\left\Vert \phi _{n}\right\Vert  \\
& \leq \max_{0\leq n\leq M}Ch^{k+1}\left\Vert u_{n}\right\Vert
_{k+1}+F\left( h,\max_{1\leq n\leq M-1}\left( k_{n}+k_{n-1}\right) \right)
\\
& =Ch^{k+1}\left\Vert \left\vert u\right\vert \right\Vert _{\infty
,k+1}+F\left( h,\max_{1\leq n\leq M-1}\left( k_{n}+k_{n-1}\right) \right),
\end{align*}%
where
\begin{equation*}
F\left( h,\max_{1\leq n\leq M-1}\left( k_{n}+k_{n-1}\right) \right) =%
\mathcal{O}\left( h^{k}+h^{s+1}+\max_{1\leq n\leq M-1}\left( k_{n}+k_{n-1}\right)
^{2}\right) .
\end{equation*}%
This concludes the proof of the first part of the theorem. \\

\noindent For second part, we have
\begin{equation*}
\sum_{n=1}^{M-1}\widehat{k}_{n}\left\Vert \nabla \left( u\left(
t_{n,\ast }\right) -u_{n,\ast }^{h}\right) \right\Vert ^{2}\leq
\sum_{n=1}^{M-1}\widehat{k}_{n}\left\Vert \nabla \left( u\left( t_{n,\ast
}\right) -u_{n,\ast }\right) \right\Vert ^{2}+\sum_{n=1}^{M-1}\widehat{%
k}_{n}\left\Vert \nabla \left( u_{n,\ast }^{h}-u_{n,\ast }\right)
\right\Vert ^{2}.
\end{equation*}%
We apply Lemma \ref{lemma:consistError} to the first term in the right hand side
\begin{equation*}
{\nu }\sum_{n=1}^{M-1}\widehat{k}_{n}\left\Vert \nabla \left( u\left(
t_{n,\ast }\right) -u_{n,\ast }\right) \right\Vert ^{2}\leq C\left( \theta
\right) {\nu }\max_{1\leq n\leq M-1}\{\left( k_{n}+k_{n-1}\right)
^{4}\}\left\Vert \nabla {u_{tt}}\right\Vert _{2,0}^{2}\ ,
\end{equation*}%
and use the triangle inequality for the second term
\begin{equation*}
{\nu }\sum_{n=1}^{M-1}\widehat{k}_{n}\left\Vert \nabla \left( u_{n,\ast
}^{h}-u_{n,\ast }\right) \right\Vert ^{2}\leq C{\nu }\sum_{n=1}^{M-1}%
\widehat{k}_{n}\left\Vert \nabla \eta _{n,\ast }\right\Vert ^{2}+C{\nu }%
\sum_{n=1}^{M-1}\widehat{k}_{n}\left\Vert \nabla \phi _{n,\ast }\right\Vert
^{2}\ .
\end{equation*}%
The last term inhere can be bound by \eqref{eq:Err-disGronwall}, while for the
first term, we use \eqref{eq:interpolation-Error} and Lemma \ref{lemma:disnormEst}
\begin{align*}
& C{\nu }\sum_{n=1}^{M-1}\widehat{k}_{n}\left\Vert \nabla \eta _{n,\ast
}\right\Vert ^{2}\leq C\left( \theta \right){\nu} \sum_{n=1}^{M-1}\left(
k_{n}+k_{n-1}\right) \left( \sum_{\ell=0}^{2}\left\Vert \nabla {\eta _{n-1+\ell}}%
\right\Vert ^{2}\right)  \\
& \leq C\left( \theta \right){\nu} h^{2k}\sum_{n=1}^{M-1}\left(
k_{n}+k_{n-1}\right) \left( \sum_{\ell=0}^{2}\left\Vert u_{n-1+\ell}\right\Vert
_{k+1}^{2}\right) \leq C\left( \theta \right) h^{2k}\left\Vert \left\vert
u\right\vert \right\Vert _{2,k+1}^{2}.
\end{align*}%
Combining the above estimates, we have
\begin{equation*}
{\nu }\sum_{n=1}^{M-1}\widehat{k}_{n}\left\Vert \nabla \left( u_{n,\ast
}^{h}-u_{n,\ast }\right) \right\Vert ^{2}\leq C\left( \theta \right){\nu}
h^{2k}\left\Vert \left\vert u\right\vert \right\Vert _{2,k+1}^{2}+C\left(
\theta \right) \widetilde{F}\left( h,\max_{1\leq n\leq M-1}\left(
k_{n}+k_{n-1}\right) \right) .
\end{equation*}%
Finally
\begin{gather}
\left( {\nu }\sum_{n=1}^{M-1}\widehat{k}_{n}\left\Vert \nabla \left(
u\left( t_{n,\ast }\right) -u_{n,\ast }^{h}\right) \right\Vert ^{2}\right) ^{%
\frac{1}{2}} \leq C{\nu }^{\frac{1}{2}}\max_{1\leq n\leq M-1}\{\left(
k_{n}+k_{n-1}\right) ^{2}\}\Vert \nabla {u_{tt}}\Vert _{2,0} \notag \\
+C{\nu }^{\frac{1}{2}}h^{k}\left\Vert \left\vert u\right\vert
\right\Vert _{2,k+1}+F\left( h,\max_{1\leq n\leq M-1}\left(
k_{n}+k_{n-1}\right) \right) , \notag
\end{gather}%
which concludes the proof of second part of the theorem.
\end{proof}

\section{Numerical Tests}

\label{section:numericaltests}

In this section, FreeFem++ is used for numerical tests with Taylor-Hood $%
(P2-P1)$ finite elements. We verify the second-order convergence and stability
of the DLN algorithm with variable time steps through three numerical experiments.

\subsection{Convergence Test (constant timestep size)}

\ The second order convergence of DLN algorithm is verified on the Taylor-Green
benchmark problem, Dyke \cite{VD}. In the domain $\Omega=(0,1)\times(0,1)$, the true
solution is
\begin{align*}
& u_{1}(x,y,t)=-\cos(w\pi x)\sin(w \pi y)\exp(-2w^{2}\pi^{2}t/\tau), \\
& u_{2}(x,y,t)=\sin(w\pi x)\cos(w \pi y)\exp(-2w^{2}\pi^{2}t/\tau), \\
& p(x,y,t)=-\frac{1}{4}(\cos(2w\pi x)+\cos(2w\pi
y))\exp(-4w^{2}\pi^{2}t/\tau),
\end{align*}
and we take the final time $T=1$, $w=1$ and $\tau=Re=100$. The body force $%
f$, initial condition, and boundary condition are determined by the true
solution. Setting $\Delta t=h$ to calculate the convergence order $R$ by the
error $e$ at two successive values of $\Delta t$ via
\begin{equation*}
R=\ln(e(\Delta t_{1})/e(\Delta t_{2}))/\ln(\Delta t_{1}/\Delta t_{2}).
\end{equation*}

\begin{table}[ptb]
\caption{The errors and convergence order of the DLN scheme at time $T=1$
for the velocity and pressure of $L^{2}$-norm with $\protect\theta=0.2$.}
\label{table:1}\centering
\begin{tabular}{clclclc}
\hline
$h=\Delta t$ & $\||e_{u}|\| _{2,0}$ & $R$ & $\||\nabla e_{u}|\|_{2,0}$
& $R$ & $\||e_{p}|\|_{2,0}$ & $R$ \\ \hline
$\frac{1}{16}$ & 0.000740428 & - & 0.0610604 & - & 0.00169375 & - \\
$\frac{1}{24}$ & 0.000228828 & 2.89 & 0.0271831 & 1.99 & 0.000687042 & 2.23
\\
$\frac{1}{32}$ & 8.89412e-05 & 3.28 & 0.0141961 & 2.26 & 0.000359889 & 2.25
\\
$\frac{1}{40}$ & 4.65027e-05 & 2.91 & 0.00912596 & 1.98 & 0.000220769 & 2.19
\\
$\frac{1}{48}$ & 2.86044e-05 & 2.67 & 0.00654533 & 1.82 & 0.000152877 & 2.02
\\
$\frac{1}{56}$ & 1.67658e-05 & 3.46 & 0.00452741 & 2.39 & 0.000107064 & 2.31
\\ \hline
\end{tabular}
\newline
\end{table}

\begin{table}[ptb]
\caption{The errors and convergence order of the DLN scheme at time $T=1$
for the velocity and pressure of $L^{\infty}$-norm with $\protect\theta=0.2$.
}
\label{table:2}\centering
\begin{tabular}{clclclc}
\hline
$h=\Delta t$ & $\||e_{u}|\|_{\infty}$ & $R$ & $\||\nabla e_{u}|\|_{\infty}$
& $R$ & $\||e_{p}|\|_{\infty}$ & $R$ \\ \hline
$\frac{1}{16}$ & 0.00122596 & - & 0.101825 & - & 0.00254809 & - \\
$\frac{1}{24}$ & 0.000399952 & 2.76 & 0.047497 & 1.88 & 0.00113562 & 1.99 \\
$\frac{1}{32}$ & 0.000162022 & 3.14 & 0.025876 & 2.11 & 0.000638476 & 2.00
\\
$\frac{1}{40}$ & 8.71029e-05 & 2.78 & 0.017116 & 1.85 & 0.000408904 & 1.99
\\
$\frac{1}{48}$ & 5.43775e-05 & 2.58 & 0.0125455 & 1.70 & 0.000291014 & 1.86
\\
$\frac{1}{56}$ & 3.24237e-05 & 3.35 & 0.00883734 & 2.27 & 0.000210233 & 2.11
\\ \hline
\end{tabular}
\newline
\end{table}

\begin{table}[ptb]
\caption{The errors and convergence order of the DLN scheme at time $T=1$
for the velocity and pressure of $L^{2}$-norm with $\protect\theta=0.5$.}
\label{table:3}\centering
\begin{tabular}{clclclc}
\hline
$h=\Delta t$ & $\||e_{u}|\|_{2,0}$ & $R$ & $\||\nabla e_{u}|\|_{2,0}$ & $R$
& $\||e_{p}|\|_{2,0}$ & $R$ \\ \hline
$\frac{1}{16}$ & 0.000700594 & - & 0.0570129 & - & 0.00134003 & - \\
$\frac{1}{24}$ & 0.000217831 & 2.88 & 0.0255791 & 1.98 & 0.000560912 & 2.11
\\
$\frac{1}{32}$ & 8.53722e-05 & 3.26 & 0.0135313 & 2.21 & 0.000305539 & 2.16
\\
$\frac{1}{40}$ & 4.50219e-05 & 2.87 & 0.00879805 & 1.93 & 0.000191838 & 2.08
\\
$\frac{1}{48}$ & 2.78268e-05 & 2.64 & 0.00634477 & 1.79 & 0.000135402 & 1.91
\\
$\frac{1}{56}$ & 1.63621e-05 & 3.44 & 0.00440779 & 2.36 & 9.57885e-05 & 2.24
\\ \hline
\end{tabular}
\newline
\end{table}

\begin{table}[ptb]
\caption{The errors and convergence order of the DLN scheme at time $T=1$
for the velocity and pressure of $L^{\infty}$-norm with $\protect\theta=0.5$.
}
\label{table:4}\centering
\begin{tabular}{clclclc}
\hline
$h=\Delta t$ & $\||e_{u}|\|_{\infty}$ & $R$ & $\||\nabla e_{u}|\|_{\infty}$
& $R$ & $\||e_{p}|\|_{\infty}$ & $R$ \\ \hline
$\frac{1}{16}$ & 0.00110053 & - & 0.0898315 & - & 0.00236018 & - \\
$\frac{1}{24}$ & 0.000354163 & 2.79 & 0.0434666 & 1.79 & 0.00105671 & 1.98
\\
$\frac{1}{32}$ & 0.000147375 & 3.05 & 0.0241532 & 2.04 & 0.000595252 & 1.99
\\
$\frac{1}{40}$ & 8.04838e-05 & 2.71 & 0.0160898 & 1.82 & 0.000381558 & 1.99
\\
$\frac{1}{48}$ & 5.0769e-05 & 2.53 & 0.011827 & 1.69 & 0.000271851 & 1.86 \\
$\frac{1}{56}$ & 3.04708e-05 & 3.31 & 0.00835234 & 2.26 & 0.000196439 & 2.11
\\ \hline
\end{tabular}
\newline
\end{table}

\begin{table}[ptb]
\caption{The errors and convergence order of the DLN scheme at time $T=1$
for the velocity and pressure of $L^{2}$-norm with $\protect\theta=0.7$.}
\label{table:5}\centering
\begin{tabular}{clclclc}
\hline
$h=\Delta t$ & $\||e_{u}|\|_{2,0}$ & $R$ & $\||\nabla e_{u}|\|_{2,0}$ & $R$
& $\||e_{p}|\|_{2,0}$ & $R$ \\ \hline
$\frac{1}{16}$ & 0.000689478 & - & 0.0560293 & - & 0.00127634 & - \\
$\frac{1}{24}$ & 0.000215154 & 2.87 & 0.025242 & 1.97 & 0.000549689 & 2.08
\\
$\frac{1}{32}$ & 8.45301e-05 & 3.25 & 0.0133912 & 2.20 & 0.000296992 & 2.14
\\
$\frac{1}{40}$ & 4.46583e-05 & 2.86 & 0.00872444 & 1.92 & 0.000187373 & 2.06
\\
$\frac{1}{48}$ & 2.76364e-05 & 2.63 & 0.00629981 & 1.79 & 0.000132745 & 1.89
\\
$\frac{1}{56}$ & 1.62635e-05 & 3.44 & 0.00438056 & 2.36 & 9.40928e-05 & 2.23
\\ \hline
\end{tabular}
\newline
\end{table}

\begin{table}[ptb]
\caption{The errors and convergence order of the DLN scheme at time $T=1$
for the velocity and pressure of $L^{\infty}$-norm with $\protect\theta=0.7$.
}
\label{table:6}\centering
\begin{tabular}{clclclc}
\hline
$h=\Delta t$ & $\||e_{u}|\|_{\infty}$ & $R$ & $\||\nabla e_{u}|\|_{\infty}$
& $R$ & $\||e_{p}|\|_{\infty}$ & $R$ \\ \hline
$\frac{1}{16}$ & 0.00101829 & - & 0.0878696 & - & 0.00241273 & - \\
$\frac{1}{24}$ & 0.000349287 & 2.64 & 0.0431141 & 1.76 & 0.00108285 & 1.98
\\
$\frac{1}{32}$ & 0.000146272 & 3.03 & 0.0240831 & 2.02 & 0.000611728 & 1.99
\\
$\frac{1}{40}$ & 8.01746e-05 & 2.69 & 0.0160849 & 1.81 & 0.000392496 & 1.99
\\
$\frac{1}{48}$ & 5.06795e-05 & 2.52 & 0.0118461 & 1.68 & 0.000279776 & 1.86
\\
$\frac{1}{56}$ & 3.05001e-05 & 3.29 & 0.00838398 & 2.24 & 0.000202406 & 2.10
\\ \hline
\end{tabular}
\newline
\end{table}

Tables \ref{table:1}, \ref{table:2}, Tables \ref{table:3}, \ref{table:4} and
Tables \ref{table:5}, \ref{table:6} correspond to $\theta=0.2, 0.5,0.7$,
respectively. The results fully verify that our DLN algorithm has
second-order convergence for both velocity and pressure, and it can be seen
that the convergence of velocity is better.

\subsection{2D Offset Circles Problem (with preset variable timestep size)}

This is a test problem from Jiang \cite{MR3260478} that is inspired by flow
between offset cylinders. The domain is a disk with a smaller off center
obstacle inside. 
Let $\Omega_{1}=\{(x,y):x^{2}+y^{2}\leq1\}$ and $\Omega_{2}=\{(x,y):(x-%
\frac {1}{2})^{2}+y^{2}\geq0.01\}$. The flow is driven by a rotational body
force:
\begin{align*}
f(x,y,t)=(-4y(1-x^{2}-y^{2}),4x(1-x^{2}-y^{2}))^{T}.
\end{align*}
with no-slip boundary conditions imposed on both circles. The body
force $f=0$ on the outer circle. The flow rotates about $(0,0)$ and the
inner circle induces a von K\'arm\'an vortex street which re-interacts
with the immersed circle creating more complex structures. Figure \ref{fig1}
and Figure \ref{fig2} show this situation.

For this test, we set $Re=200$, the number of mesh points around
the inner circle $i$ and the mesh points around the outer circle $o$
to be $10$ and $40$ respectively. The
parameter $\theta=0.5$ in DLN scheme, for the variable timestep size, the
number of computations is $n=1000$. We let the timestep size changes as the
function used in Chen and Mclaughlin \cite{MR3962839} to test stability a of different method:

\begin{align*}
k_{n}=%
\begin{cases}
0.05 & \text{$0\leq n\leq10$}, \\
0.05+0.002\sin(10t_{n}) & \text{$n>10$}.%
\end{cases}%
\end{align*}

For comparison, we also solve this problem with a standard (Variable step)
BDF2 time discretization.We calculate the energy $\frac{1}{2}\Vert u\Vert^{2}$
using BDF2 and DLN algorithms respectively. Here, let the number of mesh points
on boundary of outside circle and inner circle be $o=160$ and $i=40$
respectively and timestep $k_{0}=0.05$ and $k_{n}=k_{n-1}+0.001$.
Figure \ref{fig3a} shows that when timestep $k_{n}$ increases with time
$t$, BDF2 and DLN algorithms are respectively used to calculate energy
and in Figure \ref{fig3b}, \textit{energy of BDF2} increases with increasing
timestep, while the \textit{energy of the approximation by DLN} remains almost
constant. This verifies that the DLN  algorithm has greater stability.

\begin{figure}[ptbh]
\centering
\par
\subfigure[]{
\begin{minipage}[t]{0.33\linewidth}
\centering
\includegraphics[width=1.651in]{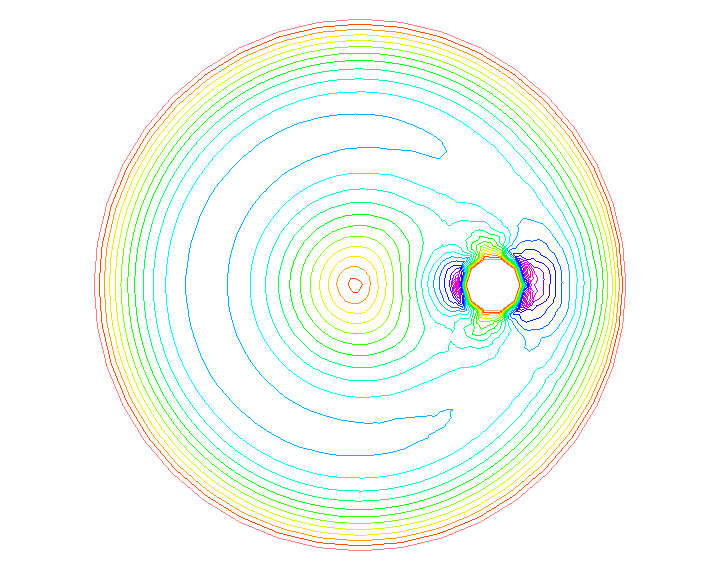}\\
\vspace{0.02cm}
\end{minipage}}%
\subfigure[]{
\begin{minipage}[t]{0.33\linewidth}
\centering
\includegraphics[width=1.651in]{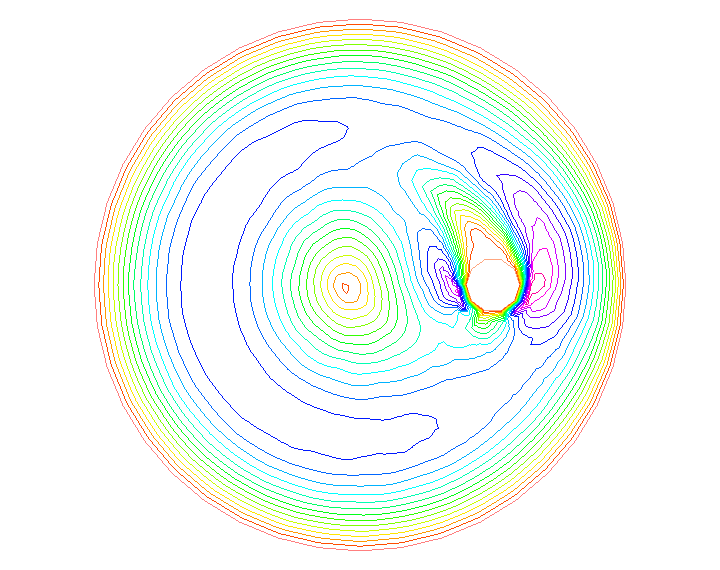}\\
\vspace{0.02cm}
\end{minipage}}\newline
\subfigure[]{
\begin{minipage}[t]{0.33\linewidth}
\centering
\includegraphics[width=1.651in]{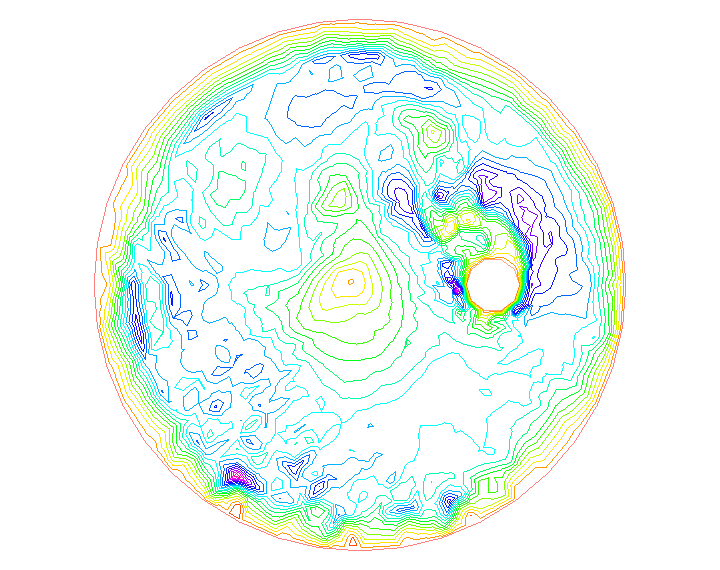}\\
\vspace{0.02cm}
\end{minipage}}%
\subfigure[]{
\begin{minipage}[t]{0.33\linewidth}
\centering
\includegraphics[width=1.651in]{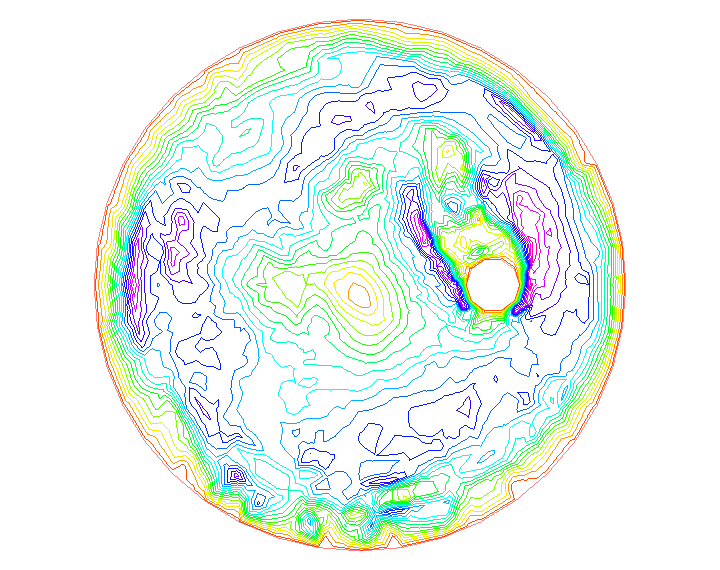}\\
\vspace{0.02cm}
\end{minipage}}
\par
\centering
\vspace{-0.2cm}
\caption{Spreed Contours of DLN. }
\label{fig1}
\end{figure}

\begin{figure}[ptbh]
\centering
\par
\subfigure[]{
\begin{minipage}[t]{0.33\linewidth}
\centering
\includegraphics[width=1.651in]{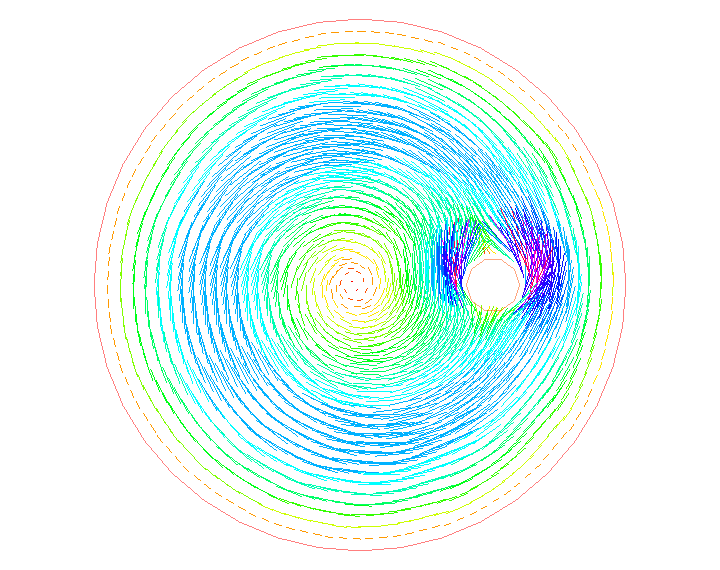}\\
\vspace{0.02cm}
\end{minipage}}%
\subfigure[]{
\begin{minipage}[t]{0.33\linewidth}
\centering
\includegraphics[width=1.651in]{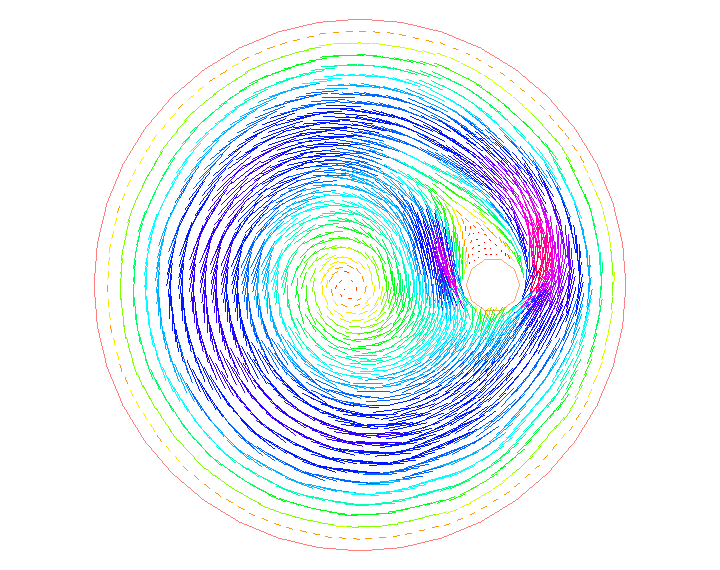}\\
\vspace{0.02cm}
\end{minipage}}\newline
\subfigure[]{
\begin{minipage}[t]{0.33\linewidth}
\centering
\includegraphics[width=1.651in]{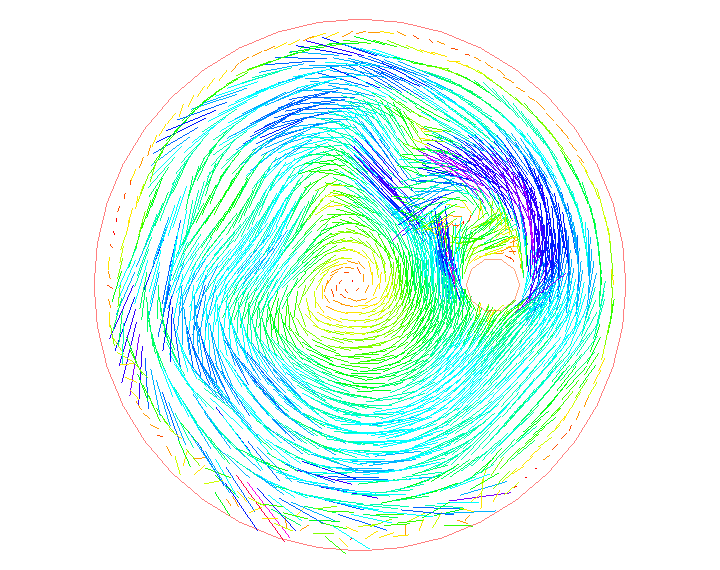}\\
\vspace{0.02cm}
\end{minipage}}%
\subfigure[]{
\begin{minipage}[t]{0.33\linewidth}
\centering
\includegraphics[width=1.651in]{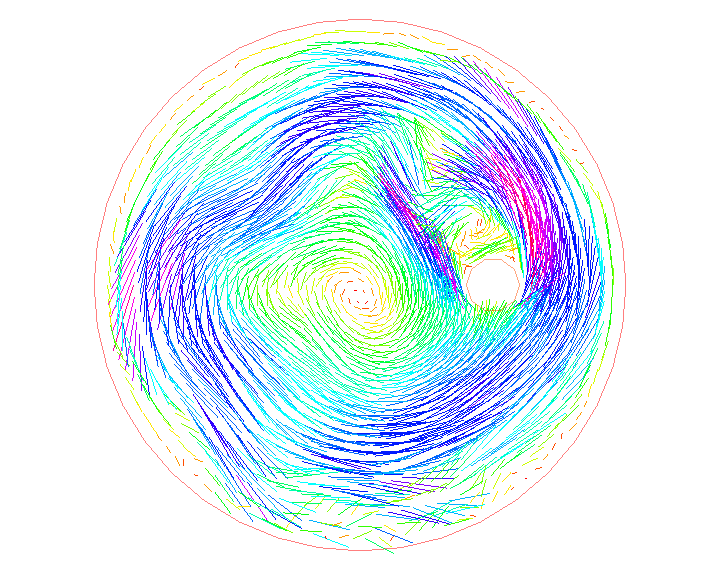}\\
\vspace{0.02cm}
\end{minipage}}
\par
\centering
\vspace{-0.2cm}
\caption{Velocity Streamlines of DLN. }
\label{fig2}
\end{figure}

\begin{figure}[ptbh]
\centering
\par
\subfigure[]{       \label{fig3a}
		\begin{minipage}[t]{0.45\linewidth}
			\centering
			\includegraphics[width=2.3in,height=1.8in]{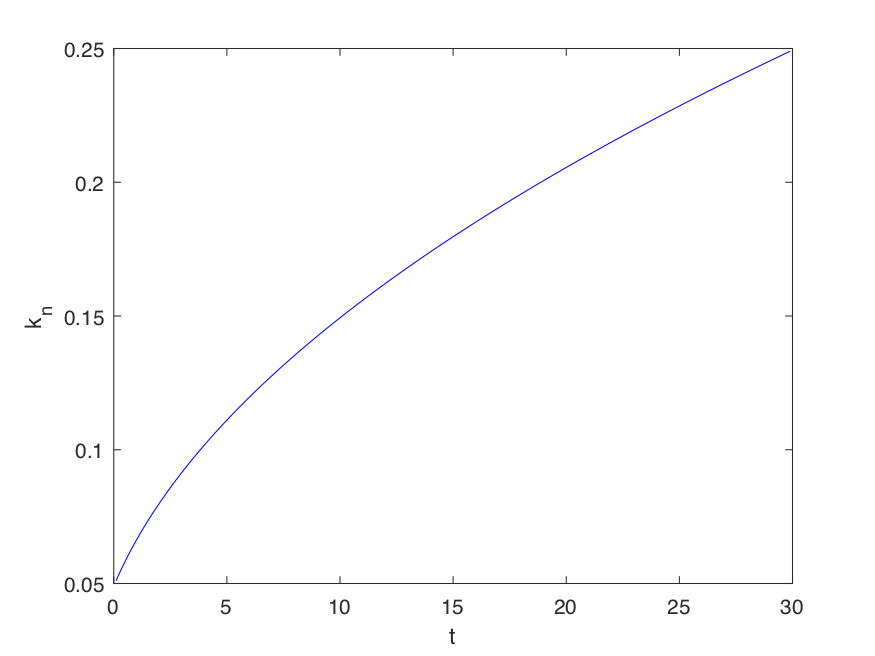}\\
			\vspace{0.02cm}
					\end{minipage}	}
\subfigure[]{       \label{fig3b}
		\begin{minipage}[t]{0.45\linewidth}
			\centering
			\includegraphics[width=2.3in,height=1.8in]{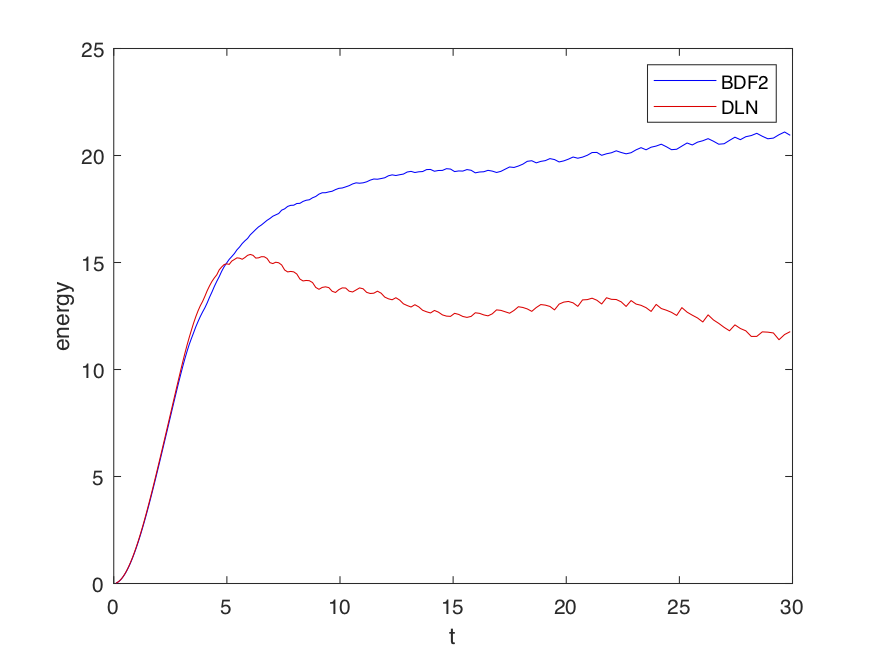}\\
			\vspace{0.02cm}
					\end{minipage}	} \centering
\vspace{-0.2cm}
\caption{Energy of DLN and BDF2 with variable timestep. }
\label{fig3}
\end{figure}

\subsection{Adapting the timestep}

Finally we use this example to perform a simple adaptivity experiment. For
this test, we adapt the timestep using the minimum dissipation criteria of
Capuano, Sanderse, De Angelis and Coppola \cite{Capuano2017AMT}. Our goal is to
test if adapting the timestep produces a significant difference in the solution.
Other criteria/estimators are under study. Their idea is to adapt the timestep
to keep the numerical dissipation, $\epsilon^{DLN}$ from the dominating physical
dissipation, $\epsilon^{\nu}$. Thus we adapt for
\begin{align*}
\chi=\left| \frac{\epsilon^{DLN}}{\epsilon^{\nu}} \right|<\delta.
\end{align*}
Here $\epsilon^{DLN}$ is the numerical dissipation and
$\epsilon^{\nu}$ is the viscous dissipation. These are given by:
\begin{align*}
&\epsilon^{DLN}=\left\|\frac{\sum_{\ell=0}^{2}a_{\ell}^{n}u_{n-1+\ell}^h}
{\sqrt{\hat{k_{n}}}} \right\|^2,\\
&\epsilon^{\nu}=\nu \left\|\nabla u_{n,\ast}^h \right\|^2.
\end{align*}
In the test, we set the tolerance for the dissipation ratio $\delta$ to be $0.002$.
The time stepsize is then adapted by halving or doubling according to
\begin{align*}
&\Delta t^{n+1}=2*\Delta t^n;\quad if \ \chi <\delta, \\
&\Delta t^{n}=0.5*\Delta t^n;\qquad if\  \chi \geq\delta.
\end{align*}
We adapted the next timestep when the dissipation ratio was out of range.
Naturally, other strategies for varying $\Delta t$ could be tested, such as
formula (16) p.2317 of Capuano, Sanderse, De Angelis and Coppola
\cite{Capuano2017AMT}. We select the final time $T = 63.7$ and minimal time stepsize to be
$0.01$. The adaptive algorithm completed in $6000$ steps. Figure \ref{fig4} and
Figure \ref{fig5} are line diagrams of time stepsize $k_{n}$, energy
$\frac{1}{2} \left\|u \right\|^{2}$, numerical dissipation $\sqrt{\epsilon^{DLN}}$
and ratio $\chi$ changing with time $T$, respectively.

Then we select the same final time $T = 63.7$, the same calculated steps $6000$
and use the constant time stepstep $k = T/6000$ to calculate to obtain the line
diagram of energy $\frac{1}{2} \left\|u \right\|^{2}$, numerical dissipation
$\sqrt{\epsilon^{DLN}}$ and ratio $\chi$ changing with time $T$, See Figure \ref{fig6}
and Figure \ref{fig7}.

\begin{figure}[ptbh]
\centering
\par
\subfigure[]{    \label{fig4a}
		\begin{minipage}[t]{0.48\linewidth}
			\centering
			\includegraphics[width=2.5in,height=1.8in]{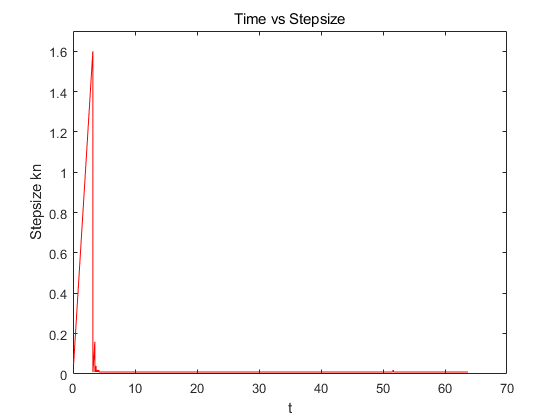}\\
			\vspace{0.02cm}
					\end{minipage}	}
\subfigure[]{    \label{fig4b}
		\begin{minipage}[t]{0.48\linewidth}
			\centering
			\includegraphics[width=2.5in,height=1.8in]{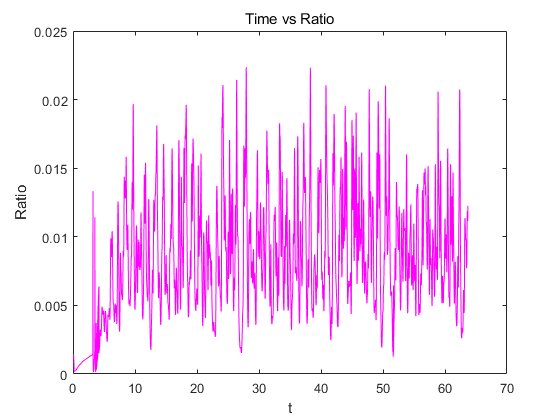}\\
			\vspace{0.02cm}
					\end{minipage}	} \centering
\vspace{-0.2cm}
\caption{The time stepsize $k_{n}$ and ratio $\chi$ changing with
adaptive time stepsize. }
\label{fig4}
\end{figure}

\begin{figure}[ptbh]
\centering
\par
\subfigure[]{       \label{fig5a}
		\begin{minipage}[t]{0.48\linewidth}
			\centering
			\includegraphics[width=2.5in,height=1.8in]{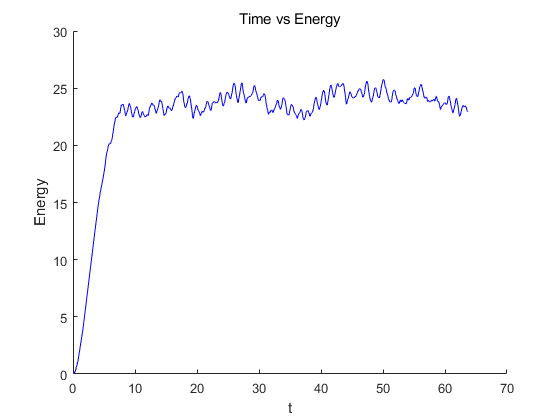}\\
			\vspace{0.02cm}
					\end{minipage}	}
\subfigure[]{       \label{fig5b}
		\begin{minipage}[t]{0.48\linewidth}
			\centering
			\includegraphics[width=2.5in,height=1.8in]{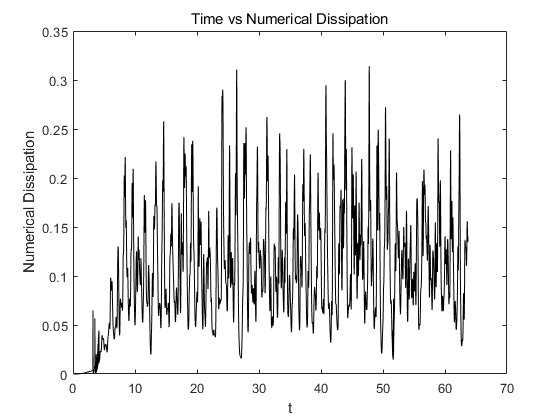}\\
			\vspace{0.02cm}
					\end{minipage}	} \centering
\vspace{-0.2cm}
\caption{The energy $\frac{1}{2} \left\| u \right\|^{2}$ and numerical
dissipation $\sqrt{\epsilon^{DLN}}$ changing with adaptive time stepsize. }
\label{fig5}
\end{figure}

\begin{figure}[ptbh]
\centering
\par
\subfigure[]{       \label{fig6a}
		\begin{minipage}[t]{0.48\linewidth}
			\centering
			\includegraphics[width=2.5in,height=1.8in]{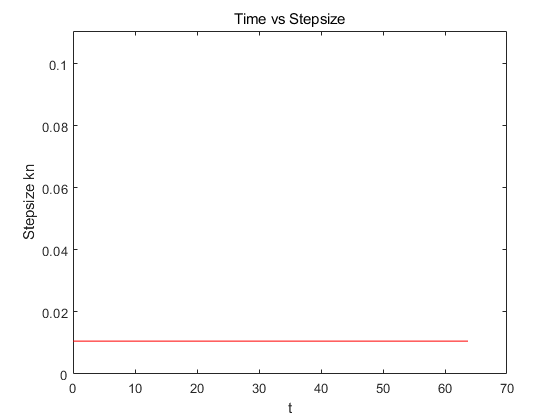}\\
			\vspace{0.02cm}
					\end{minipage}	}
\subfigure[]{       \label{fig6b}
		\begin{minipage}[t]{0.48\linewidth}
			\centering
			\includegraphics[width=2.5in,height=1.8in]{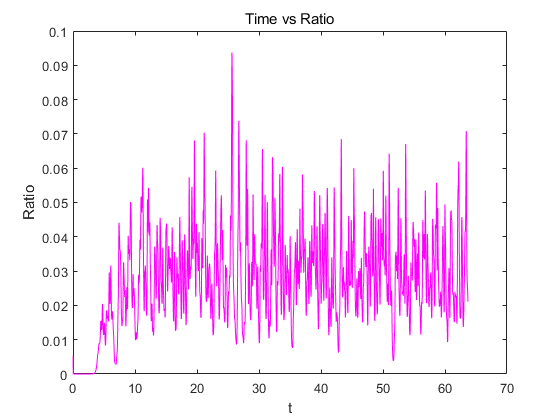}\\
			\vspace{0.02cm}
					\end{minipage}	} \centering
\vspace{-0.2cm}
\caption{The time stepsize $k$ and ratio $\chi$ changing with constant
time stepsize. }
\label{fig6}
\end{figure}

\begin{figure}[ptbh]
\centering
\par
\subfigure[]{    \label{fig7a}
		\begin{minipage}[t]{0.48\linewidth}
			\centering
			\includegraphics[width=2.5in,height=1.8in]{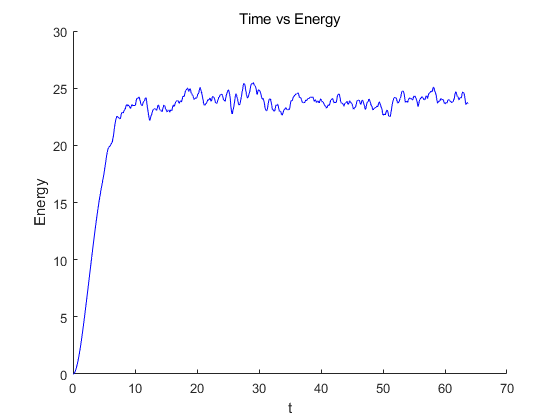}\\
			\vspace{0.02cm}
					\end{minipage}	}
\subfigure[]{    \label{fig7b}
		\begin{minipage}[t]{0.48\linewidth}
			\centering
			\includegraphics[width=2.5in,height=1.8in]{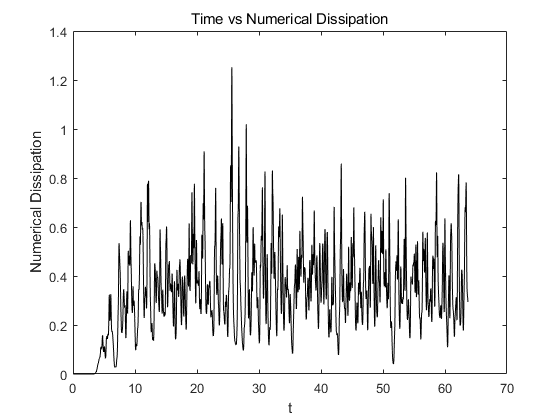}\\
			\vspace{0.02cm}
					\end{minipage}	} \centering
\vspace{-0.2cm}
\caption{The energy $\frac{1}{2} \left\| u \right\|^{2}$ and numerical
dissipation $\sqrt{\epsilon^{DLN}}$ changing with constant time stepsize. }
\label{fig7}
\end{figure}

We now compare the constant time stepsize results in Figure \ref{fig6} and
Figure \ref{fig7} with the adaptive results in Figure \ref{fig4} and Figure
\ref{fig5}. We first note that time stepsize under adaptivity reaches maximum
value $1.6$ in a few steps then goes down sharply to the minimum
stepsize $0.01$ thereafter. In the test represented in Figure \ref{fig4a},
the timestep alternates between the minimum stepsize and twice that. This
is due to the preset algorithmic choice.
DLN under constant stepsize takes 773 timesteps to reach a kinetic energy of
approximately $23$ which adaptive DLN algorithm reaches that level in $396$
timesteps. In comparison of numerical dissipation, Figure \ref{fig5b} and
\ref{fig7b} show that the numerical dissipation with adaptive time stepsize
evolves smoothly with a peak value below $0.35$. Similarly the ratio $\chi$ has
a \textit{order of magnitude smaller for adaptive time stepsize},
Figure \ref{fig4b}, than constant time stepsize, Figure \ref{fig7b}.

\nopagebreak[0]

\section{Conclusions}

Based on the theory and the simple numerical tests that for
time discretization of flow problems the 2-step DLN method is to be
preferred over the common BDF2 method. It is second order, unconditionally,
long time, nonlinearly stable. For increasing step-sizes, BDF2 injects
nonphysical kinetic energy in the discrete solution (disrupting long time
behavior and statistical equilibrium) while DLN does not. Important open
questions include how to perform error estimation in a memory and
computationally efficient (and effective) way. In particular, finding a memory
efficient estimator, as was done in Gresho, Sani and Engelman
\cite{gresho1998incompressiblevol2} for
the trapezoid rule, is a necessary step. It would be useful if the DLN
method could be embedded in a family of different orders with good
properties or if it could be induced from simpler methods by added time
filters. Both are open problems.

\bibliographystyle{amsplain}

\appendix

\section{Proof of Lemma \ref{lemma:disnormEst}}

\label{section:Est-disnorm}

\begin{proof}
For first part, if $M$ is even integer
\begin{align*}
\sum_{n=1}^{M-1} \left( k_{n} + k_{n-1} \right) \left\Vert v_{n+1}
\right\Vert _{k}^{p} \leq\sum_{\ell=1}^{M/2} \left( k_{2\ell-2} + k_{2\ell-1} \right)
\left\Vert v_{2\ell} \right\Vert _{k}^{p} \\
+ \left( k_{0} \left\Vert v_{1} \right\Vert _{k}^{p} + \sum_{\ell=1}^{M/2-1}
\left( k_{2\ell-1} + k_{2\ell} \right) \left\Vert v_{2\ell+1} \right\Vert _{k}^{p} +
k_{M-1} \left\Vert v_{M} \right\Vert _{k}^{p} \right) \\
= \left( \Vert\left| v \right| \Vert_{p,k}^{P_{1},R}\right) ^{p} + \left(
\Vert\left| v \right| \Vert_{p,k}^{P_{2},R}\right) ^{p} \ .
\end{align*}
And
\begin{align*}
\sum_{n=1}^{M-1} \left( k_{n} + k_{n-1} \right) \left\Vert v_{n-1}
\right\Vert _{k}^{p} \leq\sum_{\ell=0}^{M/2-1} \left( k_{2\ell} + k_{2\ell+1} \right)
\left\Vert v_{2\ell} \right\Vert _{k}^{p} \\
+ \left( k_{0} \left\Vert v_{0} \right\Vert _{k}^{p} + \sum_{\ell=0}^{M/2-2}
\left( k_{2\ell+1} + k_{2\ell+2} \right) \left\Vert v_{2\ell+1} \right\Vert _{k}^{p}
+ k_{M-1} \left\Vert v_{M-1} \right\Vert _{k}^{p} \right) \\
= \left( \left\Vert \left| v \right| \right\Vert _{p,k}^{P_{1},L}\right)
^{p} + \left( \left\Vert \left| v \right| \right\Vert
_{p,k}^{P_{2},L}\right) ^{p} \ .
\end{align*}
If $M$ is odd integer
\begin{align*}
\sum_{n=1}^{M-1} \left( k_{n} + k_{n-1} \right) \left\Vert v_{n+1}
\right\Vert _{k}^{p} \leq\left( \sum_{\ell=1}^{\left( M-1 \right) /2} \left(
k_{2\ell-2} + k_{2\ell-1} \right) \left\Vert v_{2\ell} \right\Vert _{k}^{p} + k_{M-1}
\left\Vert v_{M} \right\Vert _{k}^{p} \right) \\
+ \left( k_{0} \left\Vert v_{1} \right\Vert _{k}^{p} + \sum_{\ell=1}^{\left(
M-1 \right) /2} \left( k_{2\ell-1} + k_{2\ell} \right) \left\Vert v_{2\ell+1}
\right\Vert _{k}^{p} \right) = \left( \Vert\left| v \right| \Vert
_{p,k}^{P_{1},R}\right) ^{p} + \left( \Vert\left| v \right| \Vert
_{p,k}^{P_{2},R}\right) ^{p} \ .
\end{align*}
And
\begin{align*}
\sum_{n=1}^{M-1} \left( k_{n} + k_{n-1} \right) \left\Vert v_{n-1}
\right\Vert _{k}^{p} \leq\left( \sum_{\ell=1}^{\left( M-1 \right) /2} \left(
k_{2\ell-2} + k_{2\ell-1} \right) \left\Vert v_{2\ell-2} \right\Vert _{k}^{p} + k_{M-1}
\left\Vert v_{M-1} \right\Vert _{k}^{p} \right) \\
+ \left( k_{0} \left\Vert v_{0} \right\Vert _{k}^{p} + \sum_{\ell=1}^{\left(
M-1 \right) /2} \left( k_{2\ell-1} + k_{2\ell} \right) \left\Vert v_{2\ell-1}
\right\Vert _{k}^{p} \right) = \left( \Vert\left| v \right| \Vert
_{p,k}^{P_{1},L}\right) ^{p} + \left( \Vert\left| v \right| \Vert
_{p,k}^{P_{2},L}\right) ^{p} \ .
\end{align*}
\noindent It's easy to check
\begin{align*}
\sum_{n=1}^{M-1} \left( k_{n} + k_{n-1} \right) \left\Vert v_{n} \right\Vert
_{k}^{p} = \left( \Vert\left| v \right| \Vert_{p,k}^{P_{0},L}\right) ^{p}
+ \left( \Vert\left| v \right| \Vert_{p,k}^{P_{0},R} \right) ^{p} \ .
\end{align*}
\noindent Thus we have proved the first part. For second part, we can check
\begin{align*}
\sum_{n=1}^{M-1} \left( k_{n} + k_{n-1} \right) \left\Vert v \left( t_{n,*}
\right) \right\Vert _{k}^{p} = \left( \left\Vert \left| v_{*} \right|
\right\Vert _{p,k}^{\widetilde{P}_{1}} \right) ^{p} + \left( \left\Vert
\left| v_{*} \right| \right\Vert _{p,k}^{\widetilde{P}_{2}} \right) ^{p} \ .
\end{align*}
whenever $M$ is even integer or odd integer.
\end{proof}

\end{document}